\documentclass[11pt]{amsart}
\usepackage[english]{babel}     
\usepackage[utf8]{inputenc}
\usepackage{amsmath}
\usepackage{amssymb}
\usepackage{mathrsfs}
\usepackage{stmaryrd}
\usepackage{hyperref}
\usepackage{xcolor}
\usepackage{todonotes}
\usepackage{comment}
\usepackage{enumerate}
\usepackage{tikz-cd}

\newtheorem{theorem}{Theorem}[section]
\newtheorem{theoremx}{Theorem}

\newtheorem{lemma}[theorem]{Lemma}
\newtheorem{proposition}[theorem]{Proposition}
\newtheorem{corollary}[theorem]{Corollary}

\theoremstyle{definition}
\newtheorem{definition}[theorem]{Definition}

\newtheorem{remark}[theorem]{Remark}

\theoremstyle{remark}
\numberwithin{equation}{section}

\let\epsilon\varepsilon

\newcommand{\C}			{{\operatorname{\mathbb{C}}}}
\newcommand{\R}			{{\operatorname{\mathbb{R}}}}
\newcommand{\N}			{{\operatorname{\mathbb{N}}}}

\DeclareMathOperator{\supp}{supp}

\DeclareMathOperator{\id}{id}

\DeclareMathOperator{\tr}{tr}

\newcommand{\fn}{\mathfrak{n}}
\newcommand{\fm}{\mathfrak{m}}

\newcommand{\cN}{\mathcal{N}}

\newcommand{\cR}{\mathcal{R}}
\newcommand{\cE}{\mathcal{E}}
\newcommand{\cT}{\mathcal{T}}

\newcommand{\SL}{\textrm{SL}}

\newcommand{\cV}{\mathcal{V}}
\newcommand{\cM}{\mathcal{M}}
\newcommand{\cL}{\mathcal{L}}

\renewcommand{\a}{\alpha}

\newcommand{\g}{\gamma}

\newcommand{\vphi}{\varphi}
\newcommand{\z}{\zeta}
\renewcommand{\th}{\theta}

\renewcommand{\l}{\lambda}

\renewcommand{\k}{\kappa}

\newcommand{\s}{\sigma}

\newcommand{\la}{\langle}
\newcommand{\ra}{\rangle}

\title[Multilinear transference for non-unimodular groups]{Transference  of multilinear Fourier and Schur multipliers acting on non-commutative $L_p$-spaces for non-unimodular groups}

\date{\noindent \today.  MSC2010 keywords: 22D25, 43A15,  46L51.  GV is supported by the NWO Vidi grant VI.Vidi.192.018 `Non-commutative harmonic analysis and rigidity of operator algebras'.  }

\author{Gerrit Vos}

\address{TU Delft, EWI/DIAM,
	P.O.Box 5031,
	2600 GA Delft,
	The Netherlands}

 \email{G.M.Vos@tudelft.nl}

\begin{document}

\maketitle

\begin{abstract}
In \cite{CKV}, transference results between multilinear Fourier and Schur multipliers on noncommutative $L_p$-spaces were shown for unimodular groups. We propose a suitable extension of the definition of multilinear Fourier multipliers for non-unimodular groups, and show that the aforementioned transference results also hold in this more general setting.
\end{abstract}

\section{Introduction}

A central problem in classical harmonic analysis is finding conditions on symbols $\phi \in L_\infty(\R^n)$ such that the associated Fourier multiplier $T_\phi$ is bounded on $L_p(\R^n)$. One can replace $\R^n$ here by any locally compact abelian group in a straightforward way. For non-abelian groups $G$, there is no Pontryagin dual; instead, the Fourier multiplier corresponding to a function $\phi$ on $G$ is a map on the group von Neumann algebra. It is given by $\l_s \mapsto \phi(s) \l_s$ for $s \in G$, where $\l$ is the left regular representation. Equivalently, it is given by $\l(f) \mapsto \l(\phi f)$ for $f \in L_1(G)$. It turns out that symbols $\phi \in L_\infty(G)$ give rise to bounded Fourier multipliers on $\cL G$ precisely when $\phi$ defines a multiplier on the Fourier algebra $A(G)$, which coincides with the predual of $\cL G$.  \\

Another interesting question is which symbols $\phi \in L_\infty(G)$ give rise to a completely bounded multiplier on $\cL G$. Bozejko and Fendler \cite{BozejkoFendler} showed that this happens exactly when $\phi$ defines a bounded Schur multiplier of Toeplitz type on $B(L_2(G))$, and in that case the completely bounded norms are equal. This is called a transference result between Fourier and Schur multipliers. A different proof of this transference result was later provided by Jolissaint \cite{Jolissaint}. This relation between Fourier and Schur multipliers has been an important tool to prove several multiplier results. For instance, bounding the norm of Fourier multipliers by that of Schur multipliers played a crucial role in \cite{PRS}. The converse transference was used in \cite{Pisier98} to give examples of bounded multipliers on $L_p$-spaces that are not completely bounded. Similar transference techniques were used in \cite{CGPT} to prove H\"ormander-Mikhlin criteria for the boundedness of Schur multipliers, and in \cite{LafforgueDeLaSalle} to find examples of non-commutative $L_p$-spaces without the completely bounded approximation property. \\

There are several papers treating transference results for the noncommutative $L_p$-spaces $L_p(\cL G)$. Neuwirth and Ricard studied this question for discrete groups \cite{NeuwirthRicard}. In this case, the Fourier multipliers are relatively straightforward to define on $L_p(\cL G)$. They proved that if $\phi$ defines a completely bounded Fourier multiplier on $L_p(\cL G)$, then it defines a completely bounded Schur multiplier on the Schatten class $S_p(L_2(G))$. Moreover, they proved that the converse implication also holds provided that the group $G$ is amenable. Later, Caspers and De la Salle \cite{CaspersDeLaSalle} defined Fourier multipliers on the noncommutative $L_p$-spaces of general locally compact groups and proved that the same transference results also hold here. An analogous result was proved by Gonzalez-Perez \cite{gonzalez2018crossed} for crossed products. \\

Juschenko, Todorov and Turowska \cite{JTT} introduced multilinear Schur multipliers with respect to measure spaces. Such multipliers and the related notion of multiple operator integrals have been used to prove several interesting results such as the resolution of Koplienko's conjecture on higher order spectral shift functions in \cite{PSS13}. Therefore, it is also interesting to consider transference between multilinear Fourier and Schur multipliers. Todorov and Turowska \cite{TodorovTurowska} defined a multidimensional Fourier algebra and proved a transference result for multiplicatively bounded multilinear Fourier and Schur multipliers. \\

To consider multilinear results on noncommutative $L_p$-spaces, one needs a $(p_1, \dots, p_n)$-version of multiplicative boundedness. This was introduced by Caspers, Janssens, Krishnaswamy-Usha and Miaskiwskyi \cite{CJKM} along the lines of Pisier's characterisation of completely bounded norms in terms of Schatten norms \cite[Lemma 1.7]{Pisier98}. Moreover, they proved a bilinear transference result for discrete groups `along the way' in \cite[Proof of Theorem 7.2]{CJKM}, in order to provide examples of $L_p$-multipliers for semidirect products of groups. Multilinear transference was studied in a more general sense by Caspers, Krishnaswamy-Usha and the author of the present manuscript in \cite{CKV}, but only for unimodular groups. They obtained a generalisation of the linear transference results. As a direct consequence of this transference result, a De Leeuw-type restriction theorem was proven for the multiplicatively bounded norms. In this paper, we complete the picture by proving a multilinear transference result for general locally compact groups. \\

The main difficulty for non-unimodular groups comes from the fact that the Plancherel weight is not tracial. This means that we have to deal with spatial derivatives, which in our case will just be the multiplication operator with the modular function. It will be denoted by $\Delta$. In particular, this raises the question how the multilinear Fourier multiplier should be defined for $p < \infty$. It turns out that, in order to prove transference results, one needs to use the definition that `leaves the $\Delta$'s in place'. More precisely, for `suitable' $f_i$ (this will be defined later) and $x_i = \Delta^{\frac1{2p_i}} \l(f_i) \Delta^{\frac1{2p_i}}$, the Fourier multiplier is defined as 
\[
    T_\phi(x_1, \dots, x_n) = \int_{G^{\times n}} \phi(s_1, \dots, s_n) f_1(s_1) \ldots f_n(s_n) \Delta^{\frac1{2p_1}} \l_{s_1} \Delta^{\frac1{2p_1}} \ldots \Delta^{\frac1{2p_n}} \l_{s_n} \Delta^{\frac1{2p_n}} ds_1 \ldots ds_n.
\]
A major drawback of this definition is that it is not suitable for interpolation results when $n > 2$, unless the `intermediate' $p_i$'s are all equal to $\infty$, in which case it is open. All this will be discussed in Section \ref{Sect=FourierDef}. Our first main result gives the multilinear transference from Fourier multipliers as defined above to Schur multipliers. This is Theorem \ref{Thm=FouriertoSchur}. The definitions of $(p_1, \dots, p_n)$-multiplicative norms is given in Section \ref{Sect=mb}. 

\begin{theoremx} \label{Thm=FouriertoSchurIntro}
Let $G$ be a locally compact first countable group and let $1 \leq p \leq \infty$, $1<p_1, \ldots, p_n \leq \infty$ be such that $p^{-1} =  \sum_{i=1}^n p_i^{-1}$.
Let $\phi \in C_b(G^{\times n})$ and define $\widetilde{\phi} \in C_b(G^{\times n + 1})$ by
\begin{equation*}
    \widetilde{\phi}(s_0, \ldots, s_n) = \phi(s_0 s_1^{-1}, s_1 s_2^{-1}, \ldots, s_{n-1} s_n^{-1}), \qquad s_i \in G.
\end{equation*}
If $\phi$ is the symbol of a $(p_1, \ldots, p_n)$-multiplicatively bounded Fourier multiplier $T_\phi$ of $G$, then $\widetilde{\phi}$ is the symbol of a $(p_1, \ldots, p_n)$-multiplicatively bounded Schur multiplier $M_{\widetilde{\phi}}$ of $G$. Moreover,
\[
\begin{split}
& \Vert M_{\widetilde{\phi}}: S_{p_1}(L_2(G)) \times \ldots \times S_{p_n}(L_2(G)) \rightarrow S_{p}(L_2(G)) \Vert_{(p_1,\ldots,p_n)-mb}\\
& \qquad  \leq
\Vert T_{\phi}: L_{p_1}(\mathcal{L}G  ) \times \ldots \times  L_{p_n}(\mathcal{L}G ) \rightarrow L_{p}(\mathcal{L}G ) \Vert_{(p_1,\ldots,p_n)-mb}.
\end{split}
\]
\end{theoremx}

The proof is mostly an adaptation of the proof of \cite[Theorem 3.1]{CKV}. \cite[Theorem 3.1]{CKV} in turn has a large overlap with \cite[Lemma 4.6]{CJKM}. For that reason, the proof of \cite[Theorem 3.1]{CKV} only sketches the changes compared to \cite[Lemma 4.6]{CJKM}. As it seems undesirable to keep stacking sketches of changes, we have chosen to include the proof in full detail here. However, most of the work in generalising to the non-unimodular case comes from generalising the reduction lemmas \cite[Lemmas 4.3 and 4.4]{CJKM}. We do this already in Section \ref{Sect=FourierDef}. Note that \cite[Lemma 4.4]{CJKM} does not give any details for the proof, even though it is not that trivial even in the unimodular case. In Lemma \ref{Lem=Lemma4.4b} we give an elegant induction argument which fills this gap. 

Also, we need an extension of the intertwining result \cite[Proposition 3.9]{CJKM} for non-unimodular groups, which we state in Proposition \ref{Prop=Blackbox}. We will sketch the proof in a separate technical section at the end. We also note that in \cite[Theorem 3.1]{CKV}, the group was required to be second countable, but in the proof actually only first countability was needed. \\

For amenable groups, we also have the converse transference result. In fact, one no longer needs a continuous symbol, nor the first countability condition on the group. This is Corollary \ref{Cor=Schur to Fourier}.

\begin{theoremx} \label{Thm=Schur to FourierIntro}
Let $G$ be an amenable locally compact group and $1 \leq p, p_1, \dots, p_n \leq \infty$ be such that $\frac1p=\sum_{i=1}^n \frac1{p_i}$. Let $\phi \in  L_\infty(G^{\times n})$ and define $\widetilde{\phi}$ as in Theorem \ref{Thm=FouriertoSchurIntro}. If $\widetilde{\phi}$ is the symbol of a $(p_1, \dots, p_n)$-bounded (resp. multiplicatively bounded) Schur multiplier then $\phi$ is the symbol of a $(p_1, \dots, p_n)$-bounded (resp. multiplicatively bounded) Fourier multiplier. Moreover, 
\[
    \|T_\phi\|_{(p_1, \dots, p_n)} \leq \|M_{\widetilde{\phi}}\|_{(p_1, \dots, p_n)}, \qquad \|T_\phi\|_{(p_1, \dots, p_n)-mb} \leq \|M_{\widetilde{\phi}}\|_{(p_1, \dots, p_n)-mb}.
\]
\end{theoremx}
Again, the proof is similar to \cite{CKV}, but with additional technical complications. We also abstain from using ultraproduct techniques since they were not actually necessary for the proof. It should be noted that if $p_i = \infty$ for some $1 \leq i \leq n$, then our methods only yield the above boundedness results of the multilinear Fourier multiplier on $C_\l^*(G)$ in the $i$'th input (and conversely, boundedness on $C_\l^*(G)$ is all we need for the converse direction in Theorem \ref{Thm=FouriertoSchurIntro}). Of course, if $p_1 = \ldots = p_n = p = \infty$, then the result from \cite{TodorovTurowska} guarantees that the Fourier multiplier is indeed bounded on $(\cL G)^{\times n}$. 

As a result of Theorem \ref{Thm=Schur to FourierIntro}, we again get a multilinear De Leeuw-type restriction theorem.\\

Finally, we describe the structure of the paper. We start by giving the necessary preliminaries in Section \ref{Sect=Prelims}.  We also give a new definition of (linear) $p$-Fourier multipliers here. In Section \ref{Sect=FourierDef}, we discuss possible definitions of the multilinear Fourier multiplier, and explain why the definition as stated above is the correct one for transference. We also prove some properties of the multilinear Fourier multiplier that we will need later. In Sections \ref{Sect=FouriertoSchur} and \ref{Sect=SchurToFourier}, we prove the transference from Fourier to Schur (Theorem \ref{Thm=FouriertoSchurIntro}) and transference from Schur to Fourier (Theorem \ref{Thm=Schur to FourierIntro}) respectively. In Section \ref{Sect=Intertwining}, we sketch the proof of Proposition \ref{Prop=Blackbox} using Haagerup reduction. This section is rather technical and not essential to understand the bigger picture.

\section{Preliminaries} \label{Sect=Prelims}

 A more elaborate discussion of some of the statements in this section is contained in the author's PhD thesis \cite{VosThesis}.

\subsection{\texorpdfstring{$L_p$}{Lp}-spaces of group von Neumann algebras for noncommutative groups}

Let $G$ be a locally compact group, not necessarily unimodular. We will denote the left Haar measure of such a group by $\mu := \mu_G$. Its modular function will be denoted by $\Delta$. Recall that $\Delta: G \to (\R_{> 0}, \times)$ is a continuous group homomorphism satisfying 
\[
    \int_G f(s^{-1}) \Delta(s^{-1}) d\mu(s) = \int_G f(s) d\mu(s) = \Delta(t) \int_G f(st) d\mu(s), \qquad t \in G, f \in L_1(G).
\]
In the sequel, we will write $ds$ for the left Haar measure. \\

The left regular representation $\l$ acts on $L_2(G)$ by $(\l_sf)(t) = f(s^{-1}t)$. It also defines a $*$-representation of the $*$-algebra $L_1(G)$ on $L_2(G)$ by the formula
\[
    \l(f)g = \int_G f(s) (\l_sg) ds = f*g.
\]
The group von Neumann algebra $\cL G$ of $G$ is defined as 
\[
    \cL G = \{\l_s: s \in G\}'' = \{\l(f): f \in L_1(G)\}''.
\]
The group von Neumann algebra admits a canonical weight $\vphi$, named the Plancherel weight. It is defined for $x \in \cL G$ by
\[
    \vphi(x^*x) = \begin{cases} \|f\|_2^2 \qquad & \text{if } x = \l(f) \text{ for some } f \in L_2(G) \\ \infty & \text{else}. \end{cases}
\]
Here we extend $\l$ to all functions that define a bounded convolution operator. The Plancherel weight is tracial if and only if $G$ is unimodular. Similarly, there is a right regular representation $\rho$ defined by $(\rho_sf)(t) = \Delta^{1/2}(s) f(ts)$ and a Plancherel weight $\psi$ on $(\cL G)'$ defined  as $\psi(x^*x) = \|f\|^2_2$ if $x$ is given by $x = \rho(f)$ for some $f \in L_2(G)$ and $\psi(x) = \infty$ otherwise. In practice, the group $\cL G$ is often semifinite even when $G$ is not unimodular, but it is more natural to work with the Plancherel weight. For instance, the definition of $\vphi$ gives the Plancherel identity \eqref{Eqn=PlancherelIdentity} below.  \\

Let $L_p(\cL G)$ denote the Connes-Hilsum $L_p$-space corresponding to the Plancherel weight $\psi$ on $(\cL G)'$ (\cite{ConnesSpatial},\cite{Hilsum}; see also \cite{TerpLN}). To keep this paper more accessible, we will not be using Tomita-Takesaki theory or the theory of spatial derivatives, except for the last section. Instead, we will use some facts about the Connes-Hilsum $L_p$-spaces as a black box. Firstly, for $1 \leq p < \infty$, elements of $L_p(\cL G)$ are closed unbounded operators on $L_2(G)$, and sums and products of such operators are densely defined and preclosed. Addition and multiplication on $L_p(\cL G)$ are defined by taking the closures of the resulting sum resp. product, and we will use the usual addition and multiplication notations for this. For any $x \in L_p(\cL G)$, $1 \leq p \leq \infty$ we have
\begin{equation} \label{Eqn=Mult-lambda_s}
    \|\l_s x \l_t\|_{L_p(\cL G)} = \|x\|_{L_p(\cL G)}, \qquad s,t \in G.
\end{equation}

The spatial derivative $\frac{d\varphi}{d\psi}$ is just multiplication with the modular function $\Delta$; see for instance \cite[Section 3.5]{CaspersDeLaSalle} for more details. With slight abuse of notation, we will denote this operator by $\Delta$ as well. The domain of this operator is exactly the set of functions $f \in L_2(G)$ such that $\int_G \Delta^2(s) |f(s)|^2 ds < \infty$. Usually we will only apply $\Delta$ to continuous compactly supported functions, so that no technical complications can arise.\\

Now define $L := \l(C_c(G) \star C_c(G))$, where $C_c(G) \star C_c(G)$ is the linear span of elements of the form $f * g$, $f, g \in C_c(G)$. For $\theta \in [0,1]$ and $1 \leq p < \infty$, there is an embedding $\kappa^{\theta}_p: L \to L_p(\cL G)$ given by 
\[
    \kappa^{\th}_p(x) = \Delta^{(1-\theta)/p} x \Delta^{\theta/p}.
\]
For $p = \infty$ we simply set $\k_p^\th$ to be the identity. The images of the embeddings $\k^{\th}_p$ are dense in $L_p(\cL G)$, which will be crucial in the rest of the paper. \\

We now state some other facts for later use. From the definitions of the $L_p$-norm and the Plancherel weight we have the following Plancherel identity:
\begin{equation} \label{Eqn=PlancherelIdentity}
    \|\l(f) \Delta^{1/2}\|_{L_2(\cL G)} = \|f\|_2, \qquad f \in L_2(G) \cap L_1(G).
\end{equation}
As a sidenote, this implies that the map $C_c(G) \mapsto \k_2^1(\l(C_c(G)))$, $f \mapsto \l(f) \Delta^{1/2}$ extends to a unitary $L_2(G) \cong L_2(\cL G)$. Next, a straightforward calculation yields the following commutation formulae:
\begin{equation} \label{Eqn=CommutationFormula} 
    \Delta^z \l_s = \Delta^z(s) \l_s \Delta^z, \qquad z \in \C, s \in G
\end{equation}
and
\begin{equation} \label{Eqn=CommutationFormula2}
    \Delta^z \l(f)  = \l(\Delta^zf) \Delta^z, \qquad z \in \C, f \in C_c(G).
\end{equation}
Also, note that for $f = g * h$, $g,h \in C_c(G)$ and $z \in \C$, we have
\begin{equation} \label{Eqn=MultwithDelta}
    (\Delta^z f)(s) = \Delta^z(s) \int_G g(t) h(t^{-1}s) dt = \int_G \Delta^z(t) g(t) \Delta^z(t^{-1}s) h(t^{-1}s) dt = ((\Delta^z g) * (\Delta^z h))(s)
\end{equation}
and therefore $\Delta^z (C_c(G) \star C_c(G)) = C_c(G) \star C_c(G)$. Hence, \eqref{Eqn=CommutationFormula2} yields that 
\begin{equation} \label{Eqn=EqualEmbeddings}
    \k^{\th}_p(L) = \k_p^0(L), \qquad \forall\ \th \in [0,1].
\end{equation}

\subsection{Operator spaces and multiplicatively bounded maps} \label{Sect=mb}

Let $E_1, \dots, E_n, E$ be operator spaces and $T: E_1 \times \ldots \times E_n \to E$ a linear map. For $N \geq 1$, the \emph{multiplicative amplification} $T^{(N)}: M_N(E_1) \times \ldots M_N(E_n) \to M_N(E)$ of $T$ is defined as
\[
    T^{(N)}(\alpha_1 \otimes x_1, \dots, \alpha_n \otimes x_n) = \alpha_1 \ldots \a_n \otimes T(x_1, \dots, x_n), \qquad \a_i \in M_N(\C), x_i \in E_i
\]
and extended linearly. The map $T$ is said to be \emph{multiplicatively bounded} if 
\[
    \|T\|_{mb} := \sup_{N \geq 1} \|T^{(N)}\| < \infty.
\]
 Let us now define a notion of $p$-completely bounded maps. Set $E = L_p(\cM, \psi)$ for some von Neumann algebra $\cM$ and normal faithful semifinite weight $\psi$ on $\cM'$. Let $x \in M_N(E)$; then $x$ is closable and $[x] \in L_p(M_n(\cM), \tr_N \otimes \psi)$. Here $\tr_N$ is the trace on $M_N(\C)$. Now we set $S_p^N \otimes L_p(\cM, \psi)$ to be the space $M_n(L_p(\cM, \psi))$ equipped with the norm 
\[
    \|x\|_{S_p^N \otimes L_p(\cM, \psi)} := \|[x]\|_{L_p(M_N(\cM), \tr_N \otimes \psi)}.
\]
We say that an operator $T: L_p(\cM, \psi) \to L_p(\cM, \psi)$ is \emph{$p$-completely bounded} if 
\[
    \|T\|_{p-cb} := \sup_{N \geq 1} \|T^{(N)}: S_p^N \otimes L_p(\cM, \psi) \to S_p^N \otimes L_p(\cM, \psi)\| < \infty.
\]
The matrix norms satisfy
\begin{equation} \label{Eqn=LpCrossProduct}
    \|\alpha \otimes x\|_{S^N_p \otimes L_p(\cM, \vphi)} = \|\alpha\|_{S^N_p} \|x\|_{L_p(\cM, \vphi)}, \qquad \alpha \in S^N_p, x \in L_p(\cM, \vphi).
\end{equation}
For semifinite $\cM$, this is an easy exercise; for arbitrary von Neumann algebras, this requires some analysis on the interplay between spatial derivatives and matrices. \\

We note that the norm on $S_p^N \otimes L_p(\cM, \psi)$ does not give an operator space structure on $L_p(\cM, \psi)$, as it does not satisfy the axioms. However, if $\cM$ is semifinite, then the notion of $p$-completely bounded maps coincides with the usual notion of completely bounded maps by \cite{Pisier98}.\\

We now define a $(p_1, \dots, p_n)$-multiplicatively bounded norm in a similar manner. Let $1 \leq p_1, \dots, p_n, p \leq \infty$ with $p^{-1} = \sum_{i=1}^n p_i^{-1}$. We define a map $T: L_{p_1}(\cM) \times \ldots \times L_{p_n}(\cM) \to L_p(\cM)$ to be $(p_1, \dots, p_n)$-multiplicatively bounded if
\[
    \|T\|_{(p_1, \dots, p_n)-mb} := \sup_{m \geq 1} \|T^{(m)}: S_{p_1}^m \otimes L_{p_1}(\cM) \times \ldots \times S^m_{p_n} \otimes L_{p_n}(\cM) \to S_p^m \otimes L_p(\cM)\| < \infty.
\]
This turns out to be the correct notion to prove our transference results. We note that this time, there does not seem to be a case for which this coincides with `normal' multiplicative boundedness, as \cite[Lemma 1.7]{Pisier98} does not generalise to the multilinear case. If $\cM$ is semifinite, then the above definition does coincide with the definition of $(p_1, \dots, p_n)$-multiplicative boundedness from \cite{CJKM} and \cite{CKV}. Even in the semifinite case, it is unclear if this definition corresponds to complete boundedness of some linear map on some appropriate tensor product; see \cite[Remark 2.1]{CKV}.

\subsection{Fourier multipliers}
Again let $G$ be a locally compact group. For a bounded function $\phi: G \to \C$, the associated Fourier multiplier $T_\phi: \cL G \to \cL G$ is given for $f \in L_1(G)$ by $\l(f) \mapsto \l(\phi f)$, whenever this map extends weak-* continuously.  This definition has been extended to $L_p(\cL G)$ for general locally compact groups in \cite{CaspersDeLaSalle}. However, this was done only for symbols $\phi \in M_{cb}A(G)$, i.e. those symbols for which $T_\phi$ defines a completely bounded multiplier on $\cL G$. We give a broader definition here as preparation for the multilinear definition. Define $\overline{\cL G}_{(\g)}$ to be the space of closed densely defined $\g$-homogeneous operators on $L_2(G)$. See \cite[Section III, IV]{TerpLN} for a definition and some properties. We use here the facts that this space contains $L_p(\cL G)$ and that the right hand side of \eqref{Eqn=FMforn=1} is always in $\overline{\cL G}_{(-1/p)}$ (\cite[III.(19) and Corollary III.34]{TerpLN}).\\

For $\phi \in L_\infty(G)$, define $T_\phi: \k_p^\th(L) \to \overline{\cL G}_{(-1/p)}$ by  
\begin{equation} \label{Eqn=FMforn=1}
    T_\phi(\k_p^\th(\l(f))) = \Delta^{\frac{1-\th}p} \l(\phi f) \Delta^{\frac\th p}, \qquad f \in C_c(G) \star C_c(G),\ 1 \leq p < \infty,\ \th \in [0,1]
\end{equation}
The map $T_\phi$ does not depend on the choice of $\th$; this follows from \eqref{Eqn=EqualEmbeddings} and some manipulations using the commutation formula \eqref{Eqn=CommutationFormula2}. When the image of $T_\phi$ lies in $L_p(\cL G)$ and $T_\phi$ extends continuously to a bounded map on $L_p(\cL G)$, we say that $\phi$ defines a $p$-Fourier multiplier. When the extension is moreover $p$-completely bounded on $L_p(\cL G)$, we say $\phi$ defines a $p$-cb Fourier multiplier. For $\phi \in M_{cb}A(G)$, the definition of the Fourier multiplier coincides with that of \cite{CaspersDeLaSalle}. For $p=1$, this can be shown by using \cite[Proposition-Definition 3.5]{CaspersDeLaSalle} and \eqref{Eqn=Kernel of single operator} below together with \cite[Theorem 6.2 (ii), (iv)]{CaspersDeLaSalle}; for other $p$, this follows by interpolation.\\

Let us now state some known results about when symbols define $p$- or $p$-cb Fourier multipliers. Let $A(G)$ be the Fourier algebra; it is defined as $A(G) = \{f * \tilde{g}: f, g \in L_2(G)\}$ where $\tilde{g}(s) = \overline{g(s^{-1})}$. It is isomorphic to $L_1(\cL G)$, the predual of $\cL G$, through the pairing 
\[
    \la \psi, \l(f) \ra = \int_G \psi(s) f(s) ds, \qquad \psi \in A(G),\ f \in L_1(G).
\]
From classical theory (see e.g. \cite[Proposition 5.1.2]{KaniuthLau}) it is known that $\phi \in M(A(G))$ if and only if $\phi$ defines a $\infty$-Fourier multiplier. In the same way as in \cite[Proof of Definition-Proposition 3.5]{CaspersDeLaSalle}, one proves that in this case, $\phi$ defines a $p$-Fourier multiplier for all $1 \leq p \leq \infty$. Trivially, $A(G) \subseteq M(A(G))$, and hence $A(G)$ provides us with plenty of symbols defining $p$-Fourier multipliers; we will use this fact later on. As mentioned above, if $\phi$ is a $\infty$-cb Fourier multiplier (i.e. $\phi \in M_{cb}A(G)$) then $\phi$ is also a $p$-cb Fourier multiplier for $1 \leq p \leq \infty$. This is proven in \cite[Definition-Proposition 3.5]{CaspersDeLaSalle}.\\

Let us now turn our attention to the multilinear case. In \cite{TodorovTurowska}, $M^{cb}_nA(G)$ was defined to be the space of all symbols $\phi \in L_\infty(G^{\times n})$ such that the map
\[
    (\l_{s_1}, \dots, \l_{s_n}) \mapsto \phi(s_1, \dots, s_n) \l_{s_1 \ldots s_n}
\]
extends to a multiplicatively bounded normal map $(\cL G)^{\times n} \to \cL G$. In \cite{CKV} multilinear Fourier multipliers on the noncommutative $L_p$-spaces were defined for unimodular groups $G$ as follows. Let $\phi \in L_\infty(G^{\times n})$ and $1\leq p_1, \ldots, p_n, p < \infty$ with $p^{-1}=\sum_{i=1}^n p_i^{-1}$. Consider the map $T_\phi: L^{\times n} \to \cL G$ defined by
\[
    T_\phi(\lambda(f_1),\ldots, \lambda(f_n)) = \int_{G^{\times n}} \phi(t_1,\ldots,t_n) f_1(t_1)\ldots f_n(t_n) \lambda_{t_1\ldots t_n} dt_1 \ldots dt_n
\]
for $f_i \in C_c(G) \star C_c(G)$. If this map takes values in $L_p(\cL G)$ and extends continuously to $L_{p_1}(\cL G) \times \ldots \times L_{p_n}(\cL G)$, then we say that $\phi$ defines a $(p_1, \dots, p_n)$-Fourier multiplier. The extension is again denoted by $T_\phi$. In case $p_i = \infty$, we replace $L_{p_i}(\cL G)$ by $C_\l^*(G)$ in the $i$'th coordinate. If the extension is $(p_1, \dots, p_n$)-multiplicatively bounded, then we say that $\phi$ defines a $(p_1, \dots, p_n)$-mb Fourier multiplier.\\

This definition works only for unimodular groups if $p < \infty$, since $L$ is not contained in $L_p(\cL G)$ otherwise. In Section \ref{Sect=FourierDef}, we will give the definition for non-unimodular groups.

\begin{remark} \label{Rmk=inftymultipliers}
We note that a priori, the set of symbols of $(\infty, \dots, \infty)$-mb Fourier multipliers is smaller than $M_n^{cb}A(G)$. However, these sets are actually the same. This follows for instance from a combination of our results and \cite[Theorem 5.5]{TodorovTurowska}. One does not need the complicated machinery of Section \ref{Sect=FouriertoSchur} however. It follows already from the proof of \cite[Theorem 5.5]{TodorovTurowska}, or from the alternative proof of the Fourier to Schur direction in \cite[Proposition 2.3]{CKV}, that it suffices to require that $T_\phi$ is bounded on $(C_\l^*(G))^{\times n}$.
\end{remark}

\subsection{Schatten classes and Schur multipliers}

We denote by $S_p(H)$ the standard Schatten classes on the Hilbert space $H$. If $(X, \mu)$ is some measure space, then $S_2(L_2(X))$ can be isometrically identified with the space of kernels $L_2(X \times X)$. Through this identification, a kernel $A \in L_2(X \times X)$ corresponds to the operator $(A\xi)(s) = \int_X A(s,t) \xi(t) dt$. This should be seen as a continuous version of matrix multiplication. We will make no distinction between an operator $A$ and its kernel. For $1 \leq p \leq 2 \leq p' \leq \infty$ with $\frac1p + \frac1{p'} = 1$, the dual pairing between $S_p(L_2(X))$ and $S_{p'}(L_2(X))$ is given by 
\begin{equation} \label{Eqn=KernelsDualPairing}
    \la A, B \ra_{p,p'} = \int_{X^{\times 2}} A(s,t) B(t,s) dt ds, \qquad A \in S_p(L_2(X)), B \in S_2(L_2(X)).
\end{equation}
This assignment is extended continuously for general $B \in S_{p'}(L_2(X))$. We refer to \cite[Section 1.2]{LafforgueDeLaSalle} for more details.\\ 

For $\phi \in L_\infty(X^{\times n+1})$ the associated Schur multiplier is the multilinear map $S_2(L_2(X))\times\ldots \times S_2(L_2(X)) \to S_2(L_2(X))$ determined by
\[
M_\phi( A_1,\ldots, A_n)(t_0,t_n) = \int_{X^{\times n-1}} \phi(t_0,\ldots,t_n) A_1(t_0,t_1)A_2(t_1,t_2)\ldots A_n(t_{n-1},t_n) dt_1 \ldots dt_{n-1}.
\]
It follows by Cauchy-Schwarz and a straightforward calculation that $M_\phi$ does indeed take values in $S_2(L_2(X))$ (see for instance \cite{CKV}). Now let $1\leq p,p_1,\ldots,p_n \leq \infty$, with $p^{-1}=\sum_{i=1}^n p_i^{-1}$. Restrict $M_\phi$ in the $i$-th input to $S_2(L_2(X)) \cap S_{p_i}(L_2(X))$. Assume that this restriction maps to $S_p(L_2(X))$ and has a bounded extension to $S_{p_1}(L_2(X)) \times \ldots \times S_{p_n}(L_2(X))$. Then we say that $\phi$ defines a $(p_1, \dots, p_n)$-Schur multiplier. Its extension is again denoted by $M_\phi$. If $M_\phi$ is $(p_1, \dots, p_n)$-multiplicatively bounded, we say that $\phi$ defines a $(p_1, \dots, p_n)$-mb Schur multiplier.\\

The following theorem is \cite[Theorem 2.2]{CKV}; see also \cite[Theorem 1.19]{LafforgueDeLaSalle} and \cite[Theorem 3.1]{CaspersDeLaSalle}. It will be the starting point for the proof of Theorem \ref{Thm=FouriertoSchur}. 

\begin{theorem} \label{Thm=Finite truncation}
Let $\mu$ be a Radon measure on a locally compact space $X$, and $\phi: X^{n+1} \to \C$ a continuous function. Let $K>0$. The following are equivalent for $1 \leq p_1, \dots, p_n, p \leq \infty$:
\begin{enumerate}[(i)]
    \item $\phi$ defines a bounded Schur multiplier $S_{p_1}(L_2(X)) \times \ldots \times S_{p_n}(L_2(X)) \to S_p(L_2(X))$ with norm less than $K$.
    \item For every $\sigma$-finite measurable subset $X_0$ in $X$, $\phi$ restricts to a bounded Schur multiplier $S_{p_1}(L_2(X_0)) \times \ldots \times S_{p_n}(L_2(X_0)) \to S_p(L_2(X_0))$ with norm less than $K$.
    \item For any finite subset $F = \{s_1, \dots, s_N\} \subset X$ belonging to the support of $\mu$, the symbol $\phi|_{F^{\times (n+1)}}$ defines a bounded Schur multiplier $S_{p_1}(\ell_2(F)) \times \ldots \times S_{p_2}(\ell_2(F)) \to S_p(\ell_2(F))$ with norm less than $K$.
\end{enumerate}
The same equivalence is true for the $(p_1, \ldots, p_n)-mb$ norms.
\end{theorem}

Now let $G$ again be a locally compact group. In general one has $L_p(\cL G) \cap S_p(L_2(G)) = \{0\}$, so we cannot directly link Fourier and Schur multipliers as in the case $p = \infty$. In Section \ref{Sect=SchurToFourier} we will use the following trick from \cite{CaspersDeLaSalle} to circumvent this difficulty. Let $F \subseteq G$ be a relatively compact Borel subset of $G$ with positive measure, and $P_F: L_2(G) \to L_2(F)$, $f \mapsto 1_F f$ the orthogonal projection. Then for $x \in L_p(\cL G)$ one can \emph{formally} define the operator $P_F x P_F$, which lies in $S_p(L_2(G))$. We refer to \cite[Proposition 3.3, Theorem 5.1]{CaspersDeLaSalle} for details. \\

Let $F \subseteq G$ compact. We calculate the kernel of $P_F \Delta^a \lambda(f) \Delta^b P_F$: for $g \in C_c(G)$ and $s \in G$ we have
\[
\begin{split}
    (P_F \Delta^a \lambda(f) \Delta^b P_F g)(s) &= 1_F(s) \Delta^a(s) \int_G f(t) \Delta^b(t^{-1} s) 1_F(t^{-1}s) g(t^{-1} s) dt \\
    &= 1_F(s) \Delta^a(s) \int_G f(st) \Delta^b(t^{-1}) 1_F(t^{-1}) g(t^{-1}) dt \\
    &= 1_F(s) \Delta^a(s) \int_G f(st^{-1}) \Delta^b(t) 1_F(t) \Delta(t^{-1}) g(t)  dt.
\end{split}
\]
Hence the kernel of $P_F \Delta^a \lambda(f) \Delta^b P_F$ is given by
\begin{equation} \label{Eqn=Kernel of single operator}
    (s,t) \mapsto 1_F(s) \Delta^a(s)f(st^{-1}) \Delta^{b-1}(t) 1_F(t).
\end{equation}

\section{The definition of multilinear Fourier multipliers for non-unimodular groups} \label{Sect=FourierDef}

In this section, $G$ is an arbitrary locally compact group. Let $1 \leq p_1, \dots, p_n, p \leq \infty$ with $p^{-1} = \sum_{i=1}^n p_i^{-1}$ and $\phi \in L_\infty(G^{\times n})$. In this section we explore what a suitable definition for the Fourier multiplier $T_\phi: L_{p_1}(\cL G) \times \ldots L_{p_n}(\cL G) \to L_p(\cL G)$ might be. Our first requirement is that it must coincide with the linear definition for $n=1$, i.e. it must satisfy \eqref{Eqn=FMforn=1}. \\

Secondly, we would like the definition to be compatible with interpolation arguments. More precisely, if $T_\phi$ is bounded as a map $\cL G \times \ldots \times \cL G \to \cL G$ and as a map $L_{p_1}(\cL G) \times \ldots \times L_{p_n}(\cL G) \to L_p(\cL G)$, then it should also be bounded as map $L_{\frac{p_1}\nu}(\cL G) \times \ldots \times L_{\frac{p_n}\nu}(\cL G) \to L_{\frac{p}\nu}(\cL G)$ for all $0 < \nu < 1$. This means that the definition must be `compatible' with the definition on $(\cL G)^{\times n}$, in the sense that in each input, the maps $T_\phi$ must coincide on the intersection space of some compatible couple (with respect to some $\th$). This tells us what the Fourier multiplier should look like on the dense subsets $\k_{p_i}^\th(L)$:

\begin{definition}[``Wrong definition"] \label{Def=WrongDef}
Let $\th_1, \dots, \th_n, \th \in [0,1]$ and $x_i = \k_{p_i}^{\th_i}(\l(f_i))$ for $i = 1, \dots, n$, where $f_i \in C_c(G) * C_c(G)$. We set
\begin{equation} \label{Eqn=InterpolationRequirement}
    T^{\th_1, \dots, \th_n, \th}_{\phi, \text{int}}(x_1, \dots, x_n) = \k^{\th}_p(T_\phi(\l(f_1), \dots, \l(f_n))). 
\end{equation}
\end{definition}
Definition \ref{Def=WrongDef} might seem reasonable at first glance; it coincides with the linear definition for $n=1$, and it is the only option if we want interpolation results. However, there are several problems with Definition \ref{Def=WrongDef}. Firstly, the definition depends on the choice of embeddings, which is not an issue in the linear case. Secondly, there are several properties of multilinear Fourier multipliers on unimodular groups which do not carry over. This includes for instance \cite[Lemma 4.3 and Lemma 4.4]{CJKM}, which are crucial in the proof of the transference from Fourier to Schur multipliers. Moreover, if we want to prove an approximate intertwining property as in \eqref{Eqn=AmenableIntertwining}, Corollary \ref{Cor=NecessaryCondition} tells us that the definition of the Fourier multiplier has to `preserve products of linear multipliers', in the sense that
\[
    T_{\phi}(x_1, \dots, x_n) = T_{\phi_1}(x_1) \ldots T_{\phi_n}(x_n)
\]
whenever $\phi(s_1, \dots, s_n) = \phi_1(s_1) \ldots \phi_n(s_n)$. Defition \ref{Def=WrongDef} does not do this. This means that there is essentially no hope of proving the transference from Schur to Fourier multipliers either.\\

The above requirement on the preservation of products leads us to consider instead the following definition. Let $\th_i \in [0,1]$ and set $a_i = \frac{1-\theta_i}{p_i}$ and $b_i = \frac{\th_i}{p_i}$, so that $\k_{p_i}^{\th_i}(x) = \Delta^{a_i} x \Delta^{b_i}$. Now for $f_i \in C_c(G) \star C_c(G)$, we formally define the Fourier multiplier corresponding to $\th_1, \dots, \th_n$ by
\begin{equation} \label{Eqn=FourierDefV1}
\begin{split}
    T_{\phi,(\th_1, \dots, \th_n)}(\k^{\th_1}_{p_1}(\l(f_1)),\ldots, \k^{\th_n}_{p_n}(\l(f_n))) = &\int_{G^{\times n}} \phi(t_1,\ldots,t_n) f_1(t_1)\ldots f_n(t_n) \times \\
    & \Delta^{a_1} \lambda_{t_1} \Delta^{b_1 + a_2} \lambda_{t_2} \ldots \Delta^{b_{n-1} + a_n} \lambda_{t_n} \Delta^{b_n} dt_1 \ldots dt_n.
\end{split}
\end{equation}
A priori, it is not clear how to define the integral in \eqref{Eqn=FourierDefV1}. After all, the integrand 
\[
    H(t_1, \dots, t_n) := \phi(t_1,\ldots,t_n) f_1(t_1)\ldots f_n(t_n) \Delta^{a_1} \lambda_{t_1} \Delta^{b_1 + a_2} \lambda_{t_2} \ldots \Delta^{b_{n-1} + a_n} \lambda_{t_n} \Delta^{b_n}
\]
is a function that has unbounded operators as values. However, on closer inspection, the `unbounded part' of this operator doesn't really depend on the integration variables. Indeed, using the commutation formula \eqref{Eqn=CommutationFormula}, we can write
\[
\begin{split}
    &H(t_1, \dots, t_n) \\
    &= \phi(t_1, \dots, t_n) f_1(t_1) \ldots f_n(t_n) \Delta^{a_1}(t_1) \Delta^{a_1 + a_2 + b_1}(t_2) \ldots \Delta^{\sum_{i=1}^n a_i + \sum_{i=1}^{n-1} b_i}(t_n) \l_{t_1\ldots t_n} \Delta^{1/p} \\
    &=  \phi(t_1, \dots, t_n) (\Delta^{\beta_1} f_1)(t_1) \ldots (\Delta^{\beta_n} f_n)(t_n) \l_{t_1\ldots t_n} \cdot \Delta^{1/p}, \qquad \beta_j = \sum_{i=1}^j a_i + \sum_{i=1}^{j-1} b_i. \\
\end{split}
\]
Note here that the functions $\Delta^{\beta_i} f_i$ are still in $C_c(G) \star C_c(G)$ by \eqref{Eqn=MultwithDelta}. Hence, a more rigorous way to define the Fourier multiplier is 
\[
    T_{\phi,(\th_1, \dots, \th_n)}(\k^{\th_1}_{p_1}(\l(f_1)),\ldots, \k^{\th_n}_{p_n}(\l(f_n))) = T_\phi(\l(\Delta^{\beta_1} f_1), \ldots, \l(\Delta^{\beta_n} f_n)) \Delta^{1/p}.
\]
However, we will keep the notation from \eqref{Eqn=FourierDefV1}. The integral is justified through the above arguments. The latter expression also makes clear that \eqref{Eqn=FourierDefV1} takes values in the space of closed densely defined $(-1/p)$-homogeneous operators on $L_2(G)$ (see \cite[III.(19) and Corollary III.34]{TerpLN}). Just as in the linear case, it is not clear that \eqref{Eqn=FourierDefV1} takes values in $L_p(\cL G)$ in general; this will be part of the assumptions. \\

It turns out that the operator $T_{\phi, (\th_1, \dots, \th_n)}$ in \eqref{Eqn=FourierDefV1} does not depend on the choice of $\theta_i$'s:

\begin{proposition} \label{Prop=IndependenceOfTheta}
Let $1 \leq p_1, \dots, p_n, p < \infty$ and $\th_1, \dots, \th_n \in [0,1]$. The maps $T_{\phi,(\th_1, \dots, \th_n)}$ and $T_{\phi,(0,\dots,0)}$ coincide on the space $\k^0_{p_1}(L) \times \ldots \times \k^0_{p_n}(L)$. Consequently, if one of the maps has image in $L_p(\cL G)$ and extends continuously to $L_{p_1}(\cL G) \times \ldots \times L_{p_n}(\cL G)$, then the other does as well and these extensions are equal. 
\end{proposition}

\begin{proof}
Recall that by \eqref{Eqn=EqualEmbeddings}, $\k_{p_i}^{\th_i}(L) = \k_{p_i}^0(L)$ for $i = 1 \dots, n$. For any such $i$, take $a_i, b_i$ as above, i.e. so that $\k_{p_i}^\th(x) = \Delta^{a_i} x \Delta^{b_i}$. Let $f_i \in C_c(G) \star C_c(G)$ and set $g_i = \Delta^{-b_i} f_i$. Then by \eqref{Eqn=CommutationFormula2}, $\Delta^{a_i} \l(f_i) \Delta^{b_i} = \Delta^{1/p_1} \l(g_i) =: x_i$, i.e. $\k_{p_i}^{\th_i}(\l(f_i)) = \k_{p_i}^0(\l(g_i))$. By \eqref{Eqn=CommutationFormula} we find
\[
\begin{split}
    T_{\phi,(\th_1, \dots, \th_n)}(x_1, \dots, x_n) = &\int_{G^{\times n}} \phi(t_1,\ldots,t_n) f_1(t_1)\ldots f_n(t_n) \times \\
    & \Delta^{a_1} \lambda_{t_1} \Delta^{b_1 + a_2} \lambda_{t_2} \ldots \Delta^{b_{n-1} + a_n} \lambda_{t_n} \Delta^{b_n} dt_1 \ldots dt_n \\
    = & \int_{G^{\times n}} \phi(t_1,\ldots,t_n) \Delta^{-b_1}(t_1) f_1(t_1) \Delta^{-b_2}(t_2) f_2(t_2) \ldots \Delta^{-b_n}(t_n)f_n(t_n) \times \\
    &  \Delta^{1/p_1} \lambda_{t_1} \Delta^{1/p_2} \ldots \Delta^{1/p_n} \lambda_{t_n}  dt_1 \ldots dt_n  \\
    = &T_{\phi,(0, \dots, 0)}(\Delta^{1/p_1} \l(g_1), \dots, \Delta^{1/p_n} \l(g_n)) = T_{\phi,(0, \dots, 0)}(x_1, \dots, x_n).
\end{split}
\]
\end{proof}

With this issue out of the way, we can now define Fourier multipliers independent of the choice of $\th_i$'s:

\begin{definition}[``Correct definition"] \label{Def=CorrectDef}
Let $1 \leq p_1, \dots, p_n, p \leq \infty$ with $p^{-1} = \sum_{i=1}^n p_i^{-1}$. Also let $\phi \in L_\infty(G^{\times n})$. For $i = 1, \dots, n$, take any $a_i, b_i \in [0,1]$ such that $a_i + b_i = p_i^{-1}$. If the map 
\[
    T_\phi: \k_{p_1}^{0}(L) \times \ldots \times \k_{p_n}^0(L) \to \overline{\cL G}_{(-1/p)}
\] 
which is given for $x_i = \Delta^{a_i} \lambda(f_i) \Delta^{b_i}$ with $f_i \in C_c(G) \star C_c(G)$ by
\begin{equation} \label{Eqn=Fourier def}
\begin{split}
    T_\phi(x_1,\ldots, x_n) = &\int_{G^{\times n}} \phi(t_1,\ldots,t_n) f_1(t_1)\ldots f_n(t_n) \times \\
    & \Delta^{a_1} \lambda_{t_1} \Delta^{b_1 + a_2} \lambda_{t_2} \ldots \Delta^{b_{n-1} + a_n} \lambda_{t_n} \Delta^{b_n} dt_1 \ldots dt_n,
\end{split}
\end{equation}
takes values in $L_p(\cL G)$ and extends boundedly to $L_{p_1}(\cL G) \times \ldots \times L_{p_n}(\cL G)$ in the norm topology (in case $p_i = \infty$ for some $i$, we use the space $C_\lambda^*(G)$ instead of $L_\infty(\cL G) = \cL G$ in the $i$'th leg) then we say that $\phi$ defines a $(p_1, \dots, p_n)$-Fourier multiplier. We denote the extension by $T_\phi$, or $T_{\phi}^{p_1, \dots, p_n}$ when we wish to emphasize the domain of the operator. This is especially useful when writing an operator norm, since writing out the full domain and codomain generally makes equations too long. If $T_\phi$ is $(p_1, \dots, p_n)$-multiplicatively bounded, then we say that $\phi$ defines a $(p_1, \dots, p_n)$-mb Fourier multiplier.
\end{definition}

Clearly, for $n=1$, Definition \ref{Def=CorrectDef} reduces to \eqref{Eqn=FMforn=1}. It does not give the problems that Definition \ref{Def=WrongDef} does; as we saw already, it does not depend on the choice of embeddings. Moreover, the properties \cite[Lemma 4.3 and 4.4]{CJKM} do carry over, as we show in Lemmas \ref{Lem=Lemma 4.3} and \ref{Lem=Lemma4.4b}. Finally, it preserves products in the following more general way: if $\phi$ is such that there exist $m < n$ and $\phi_1: G^{\times m} \to \C$, $\phi_2: G^{\times n-m} \to \C$ such that
\[
    \phi(s_1, \dots, s_n) = \phi_1(s_1, \dots, s_m) \phi_2(s_{m+1}, \dots, s_n),
\]
then 
\begin{equation} \label{Eqn=FourierProducts}
    T_{\phi}(x_1, \dots, x_n) = T_{\phi_1}(x_1, \dots, x_m) T_{\phi_2}(x_{m+1}, \dots, x_n).
\end{equation} However, we have to give up interpolation results in general. The only instances where interpolation might work is when the $L_p$-spaces `in the middle' are all equal to $C_\l^*(G)$. Indeed, in that case we can take $\th_1 = 0$, $\th_n = 1$, $\th = \frac{p}{p_n}$, so that the Fourier multiplier $T_\phi$ from Definition \ref{Def=CorrectDef} also satisfies \eqref{Eqn=InterpolationRequirement}. We note that for $n > 2$ and $p_i < \infty$ for some $2 \leq i \leq n-1$, \eqref{Eqn=Fourier def} is not of the form \eqref{Eqn=InterpolationRequirement} for any $\th_1, \dots, \th_n, \th$, and hence $T_\phi$ cannot be a compatible morphism for any `usual' compatible couple structures on $(\cL G, L_{p_i}(\cL G))_{\th_i}$.

\begin{remark} \label{Rmk=InterpolationTrouble}
Although \eqref{Eqn=InterpolationRequirement} is a necessary condition for the Fourier multiplier to allow interpolation, we have not been able to prove that it is a sufficient condition. The issue is that to prove that the mapping for $(p_1, \dots, p_n)$ is compatible with the one for $(\infty, \dots, \infty)$, we have to prove that they coincide on the entire intersection space $L_p(\cL G) \cap \cL G$ (within the appropriate compatible couple structure). However, we do not know whether $L$ is dense in this space in the intersection norm. In fact, for $p > 2$, we do not even know if $\cT_\vphi^2$ is dense in the intersection norm. 
\end{remark}

\begin{remark}
We could have just taken (the extension of) the map $T_\phi^{\frac12, \dots, \frac12}$ as the definition of our Fourier multiplier. This would have allowed us to skip Proposition \ref{Prop=IndependenceOfTheta}, and all the proofs further on in this paper would still work by approximating only with elements in the central embedding. However, the more general definition allows some flexibility to choose convenient embeddings for notation or to avoid some technicalities (in particular in Lemma \ref{Lem=Lemma 4.3}).
\end{remark}

Let us now prove some properties of the multilinear Fourier multiplier for later use. Lemmas \ref{Lem=Lemma 4.3}, \ref{Lem=Lemma4.4} and \ref{Lem=Lemma4.4b} are used in the proof of Theorem \ref{Thm=FouriertoSchur}. Here Lemma \ref{Lem=Lemma 4.3} generalises \cite[Lemma 4.3]{CJKM} and Lemma \ref{Lem=Lemma4.4b} generalises \cite[Lemma 4.4]{CJKM}. Since the proofs of these two lemmas require careful bookkeeping with modular functions, we will give the full details. The proof of \cite[Lemma 4.4]{CJKM} was omitted, but it is not that trivial; our argument fills that gap.

\begin{lemma} \label{Lem=Lemma 4.3}
Let $1 \leq p_j, p \leq \infty$ and fix some $1 \leq i \leq n$. Suppose that $\phi: G^{\times n} \to \C$ is bounded and measurable and set for $r,t,r' \in G$:
\[
    \bar{\phi}(s_1, \dots, s_n; r,t,r') := \phi(rs_1, \dots, s_it, t^{-1} s_{i+1}, \dots, s_n r').
\]
Then $\phi$ defines a $(p_1, \dots, p_n)$-Fourier multiplier (resp. $(p_1, \dots, p_n)$-mb Fourier multiplier) iff $\bar{\phi}(\: \cdot \:; r,t,r')$ defines a $(p_1, \dots, p_n)$-Fourier multiplier (resp. $(p_1, \dots, p_n)$-mb Fourier multiplier). In that case, for $x_j \in L_{p_j}(\cL G)$,
\begin{equation} \label{Eqn=Lemma4.3}
    T_{\bar{\phi}(\: \cdot\:;r,t,r')}(x_1, \dots, x_n) = \l_r^* T_\phi\big(\l_r x_1, x_2, \dots, x_i \l_t, \l_t^* x_{i+1}, \dots, x_n\l_{r'}\big) \l_{r'}^*.
\end{equation}
Further, we have 
\[
    \|T^{p_1, \dots, p_n}_\phi\| = \|T^{p_1, \dots, p_n}_{\bar{\phi}(\: \cdot\:; r,t,r')}\|
\]
and $(r,t,r') \mapsto T_{\bar{\phi}(\: \cdot\:; r,t,r')}$ is strongly continuous. In the multiplicatively bounded case, we have for any $N \geq 1$
\[
    \|(T^{p_1, \dots, p_n}_\phi)^{(N)}\| = \|(T^{p_1, \dots, p_n}_{\bar{\phi}(\: \cdot\:; r,t,r')})^{(N)}\|
\]
as maps $S_{p_1}^N[L_{p_1}(\cL G)] \times \ldots \times S_{p_n}^N[L_{p_n}(\cL G)] \to S_p^N[L_p(\cL G)]$, and $(r,t,r') \mapsto T_{\bar{\phi}(\: \cdot\:; r,t,r')}^{(N)}$ is strongly continuous.
\end{lemma}

\begin{proof}
It is straightforward to check that for $s \in G$, $f \in C_c(G)$, we have $\l_s\l(f) = \l(\l_s(f)) = \l(f(s^{-1} \: \cdot \:))$; moreover, we have
\begin{equation} \label{Eqn=Lemma4.3eq1} 
    \l(f)\l_s = \int_G f(t) \l_{ts} dt = \Delta(s^{-1}) \int_G f(ts^{-1}) \l_t dt =  \Delta(s^{-1}) \l(f(\: \cdot\: s^{-1})).
\end{equation}
We will only make a choice for some of the embeddings and leave the rest open; this is notationally more convenient. Let $x_j = \Delta^{a_j} \lambda(f_j) \Delta^{b_j} \in L_{p_j}(\cL G)$, with $f_j \in C_c(G) \star C_c(G)$ and $a_1=b_i = a_{i+1} = b_n = 0$ (hence $b_1 = \frac1{p_1}, a_i = \frac1{p_i}$, etc). We compute
\[
\begin{split}
    &T_{\bar{\phi}(\: \cdot\:;r,t,r')}(x_1, \dots, x_n) \\
    = &\int_{G^{\times n}} \bar{\phi}(s_1, \dots, s_n; r,t,r') f_1(s_1) \dots f_n(s_n) \lambda_{s_1} \Delta^{b_1 + a_2} \l_{s_2} \dots \Delta^{b_{n-1} + a_n}\l_{s_n} ds_1 \dots ds_n \\
    = &\int_{G^{\times n}} \phi(s_1, \dots, s_n) f_1(r^{-1}s_1) \dots \Delta(t)^{-1} f_i(s_it^{-1}) f_{i+1}(ts_{i+1}) \dots \Delta(r')^{-1} f_n(s_n (r')^{-1}) \times \\
    & \l_{r^{-1}s_1} \Delta^{b_1 + a_2} \dots \Delta^{b_{i-1} + a_i}\l_{s_i s_{i+1}} \Delta^{b_{i+1}+ a_{i+2}} \dots \Delta^{b_{n-1} + a_n} \l_{s_n (r')^{-1}} ds_1 \dots ds_n \\
    = & \l_r^* T_\phi\big(\widetilde{x_1}, x_2, \dots, \widetilde{x_i}, \widetilde{x_{i+1}}, \dots, \widetilde{x_n}\big) \l_{r'}^*.
\end{split}
\]
Here
\[
\begin{split}
    \widetilde{x_1} &:= \lambda(f_1(r^{-1}\: \cdot\:)) \Delta^{b_1}; \qquad \widetilde{x_i} := \Delta^{-1}(t) \Delta^{a_i} \lambda(f_i(\: \cdot\: t^{-1})); \\
    \widetilde{x_{i+1}} &:= \lambda(f_{i+1}(t\:\cdot\:)) \Delta^{b_{i+1}}; \qquad \widetilde{x_n} :=  \Delta^{-1}(r')\Delta^{a_n} \lambda(f_n(\:\cdot\: (r')^{-1})).
\end{split}
\] 
By \eqref{Eqn=Lemma4.3eq1} we can write
\[
    \widetilde{x_n} = \Delta^{a_n} \l(f_n) \l_{r'} = x_n \l_{r'}
\]
and similarly
\[
    \widetilde{x_1} = \l_r x_1; \qquad \widetilde{x_i} = x_i \l_t; \qquad \widetilde{x_{i+1}} = \l_t^* x_{i+1}.
\]
Combining everything together we conclude 
\[
    T_{\bar{\phi}(\: \cdot\:;r,t,r')}(x_1, \dots, x_n) = \l_r^* T_\phi\big(\l_r x_1, x_2, \dots, x_i \l_t, \l_t^* x_{i+1}, \dots, x_n\l_{r'}\big) \l_{r'}^*.
\]
By \eqref{Eqn=Mult-lambda_s}, we have 
\[
\begin{split}
    \|T_{\bar{\phi}(\: \cdot\:;r,t,r')}(x_1, \dots, x_n)\|_p &=  \|T_\phi\big(\l_r x_1, x_2, \dots, x_i \l_t, \l_t^* x_{i+1}, \dots, x_n\l_{r'}\big)\|_p \\
    &\leq \|T^{p_1, \dots, p_n}_{\phi}\| \|x_1\|_{p_1} \ldots \|x_n\|_{p_n}.
\end{split}
\]
Hence, on the dense subsets of elements $x_j$ as above we have $\|T_{\widetilde{\phi}(\: \cdot\: ;r,t,r')}\| \leq \|T_{\phi}\|$. If we set $\psi = \tilde{\phi}(\: \cdot\:; r,t,r')$, then $\tilde{\psi}(\: \cdot\:; r^{-1}, t^{-1}, (r')^{-1}) = \phi$. Hence, applying the above result to $\tilde{\psi}(\: \cdot\:; r^{-1}, t^{-1}, (r')^{-1})$ yields the reverse inequality. By density, we conclude that the first three statements of the lemma hold. By \cite[Lemma 2.3]{JungeSherman}, the (left or right) multiplication with $\l_s$, $s \in G$ is strongly continuous in $s$. This implies the strong continuity of $(r,t,r') \mapsto T_{\tilde{\phi}(\: \cdot\:;r,t,r')}$. \\

Now assume $\phi$ defines a $(p_1, \dots, p_n)$-mb Fourier multiplier and let $N \geq 1$. Denote by $\iota_N$ the $N \times N$-identity matrix. Then by writing out the definitions and using \eqref{Eqn=Lemma4.3} we find, for $x_i \in S_{p_i}^N \otimes L_{p_i}(\cL G)$,
\[
\begin{split}
    &T_{\bar{\phi}(\cdot;r,t,r')}^{(N)}(x_1, \dots, x_n) \\
    = &(\iota_N \otimes \l_r^*) T_{\phi}^{(N)}((\iota_N \otimes \l_r) x_1, \dots, x_i(\iota_N \otimes \l_t), (\iota_N \otimes \l_t^*) x_{i+1}, \dots, x^n(\iota_N \otimes \l_{r'})) (\iota_N \otimes \l_{r'}^*).
\end{split}
\]
Hence, by a complete/matrix amplified version of the above arguments, we deduce the last two statements.
\end{proof}

\begin{lemma} \label{Lem=Lemma4.4}
Let $1 \leq p_j, p \leq \infty$ and fix some $1 \leq i \leq n$. Suppose that $\phi: G^{\times n} \to \C$ defines a $(p_1, \dots, p_n)$-Fourier multiplier and $\phi_i: G \to \C$ defines a $p_i$-Fourier multiplier. Set 
\[
\bar{\phi}(s_1, \dots, s_n) = \phi(s_1, \dots, s_n) \phi_i(s_i).
\]
Then $\bar{\phi}$ defines a $(p_1, \dots, p_n)$-Fourier multiplier and for $x_j \in L_{p_j}(\cL G)$,
\begin{equation} \label{Eqn=Lemma4.4a}
    T_{\bar{\phi}}(x_1, \dots, x_n) = T_\phi(x_1, \dots, x_{i-1}, T_{\phi_i}(x_i), x_{i+1}, \dots, x_n).
\end{equation}
In particular, 
\begin{equation} \label{Eqn=Lemma4.4abound}
    \|T_{\bar{\phi}}^{p_1, \dots, p_n}\| \leq \|T_{\phi}^{p_1, \dots, p_n}\| \|T_{\phi_i}: L_{p_i}(\cL G) \to L_{p_i}(\cL G)\|.
\end{equation}
\end{lemma}

\begin{proof}
For $x_j \in \k_{p_j}^0(L)$ (or any other embedding), it follows directly from writing out the definitions that \eqref{Eqn=Lemma4.4a} holds (cf.  \eqref{Eqn=FMforn=1}). By density, \eqref{Eqn=Lemma4.4a} holds for general $x_j \in L_{p_j}(\cL G)$ which implies \eqref{Eqn=Lemma4.4abound}, so $T_{\bar{\phi}}^{p_1, \dots, p_n}$ is bounded.
\end{proof}

\begin{lemma} \label{Lem=Lemma4.4b}
Let $1 \leq p_1, \dots, p_n \leq \infty$. Let $q_j^{-1} = \sum_{i=j}^n p_i^{-1}$ and suppose that $\phi_j: G \to \C$ defines a $q_j$-Fourier multiplier for $1 \leq j \leq n$. Set 
\[
    \bar{\phi}(s_1, \dots, s_n) = \phi_1(s_1 \ldots s_n) \phi_2(s_2 \ldots s_n) \ldots \phi_n(s_n).
\]
Then $\bar{\phi}$ defines a $(p_1, \dots p_n)$-Fourier multiplier and for $x_i \in L_{p_i}(\cL G)$ we have
\begin{equation} \label{Eqn=Lemma4.4b}
    T_{\bar{\phi}}(x_1, \dots, x_n) = T_{\phi_1}(x_1 T_{\phi_2}(x_2 \ldots T_{\phi_n}(x_n)\ldots )).
\end{equation}
\end{lemma}

\begin{proof}
We first show \eqref{Eqn=Lemma4.4b} on the dense subset $\k^0_{p_i}(L) \times \ldots \times \k^0_{p_n}(L)$. The lemma then follows from boundedness of the $T_{\phi_i}$ together with H\"olders inequality.\\

We make the slightly stronger claim that for any $\phi_2, \dots, \phi_n$ as in the assumptions and any $x_i \in \k^0_{p_i}(L)$, there exists a compactly supported function $g: G \to \C$ such that for all $\phi_1$ as in the assumptions,
\begin{equation} \label{Eqn=Lemm4.4bExtended}
    T_{\bar{\phi}}(x_1, \dots, x_n) = \Delta^{\frac1{q_1}} \l(\phi_1g) = T_{\phi_1}(x_1 T_{\phi_2}(x_2 \ldots T_{\phi_n}(x_n)\ldots)).
\end{equation}
We will prove \eqref{Eqn=Lemm4.4bExtended} with induction on $n$. We will need this intermediate step in order to expand the outer Fourier multiplier in the right-hand side of \eqref{Eqn=Lemma4.4b}. \\

The case $n=1$ follows directly from \eqref{Eqn=FMforn=1}. Now assume that \eqref{Eqn=Lemm4.4bExtended} holds for any choice of $\phi_1, \dots, \phi_{n-1}$ as above and $x_1, \dots, x_{n-1}$ with $x_i \in \k^0_{p_i'}(L)$. Fix functions $\phi_1, \dots, \phi_n$ as in the assumptions and $x_1, \dots, x_n$ so that $x_i = \Delta^{\frac1{p_i}} \l(f_i)$ for $f_i \in C_c(G) \star C_c(G)$. Take $g$ compactly supported such that, for any $\psi: G \to \C$ defining a $q_2$-Fourier multiplier,
\begin{equation} \label{Eqn=gdef}
     T_{\psi}(x_2 T_{\phi_3}(x_3 \ldots T_{\phi_n}(x_n)\ldots)) = \Delta^{\frac1{q_2}} \l(\psi g) = T_{\bar{\psi}}(x_2, \dots, x_n)
\end{equation}
where $\bar{\psi}(s_2, \dots, s_n) = \psi(s_2 \ldots s_n) \phi_3(s_3\ldots s_n) \ldots \phi_n(s_n)$. We calculate 
\[
\begin{split}
    T_{\phi_1}(x_1 T_{\phi_2}(x_2 \ldots T_{\phi_n}(x_n)\ldots)) &\stackrel{\eqref{Eqn=gdef}}= T_{\phi_1}(\Delta^{\frac1{p_1}} \l(f_1) \Delta^{\frac1{q_2}} \l(\phi_2 g)) \\
    &\stackrel{\eqref{Eqn=CommutationFormula2}}= T_{\phi_1}(\Delta^{\frac1{q_1}} \l((\Delta^{-\frac1{q_2}} f_1)*(\phi_2g))) \\
    &\stackrel{\eqref{Eqn=FMforn=1}}= \Delta^{\frac1{q_1}} \l(\phi_1 ((\Delta^{-\frac1{q_2}} f_1)*(\phi_2g))).
\end{split}
\]
This shows the second equality from \eqref{Eqn=Lemm4.4bExtended}. Continuing the previous equation, 
\[ 
\begin{split}
    T_{\phi_1}(x_1 T_{\phi_2}(x_2 \ldots T_{\phi_n}(x_n)\ldots)) &= \int_G \phi_1(t) \left(\int_G (\Delta^{-\frac1{q_2}} f_1)(s_1) (\phi_2g)(s_1^{-1} t) ds_1\right) \Delta^{\frac1{q_1}} \l_t dt \\
    &= \int_G \int_G \phi_1(s_1t) (\Delta^{-\frac1{q_2}} f_1)(s_1) (\phi_2g)(t) \Delta^{\frac1{q_1}} \l_{s_1t} dt ds_1 \\
    &\stackrel{\eqref{Eqn=CommutationFormula}}= \int_G f_1(s_1) \Delta^{\frac1{p_1}} \l_{s_1} \int_G \phi_1(s_1t) \phi_2(t) g(t) \Delta^{\frac1{q_2}} \l_t dt ds_1 \\
    &= \int_G f_1(s_1) \Delta^{\frac1{p_1}} \l_{s_1} \Delta^{\frac1{q_2}} \l(\phi_1(s_1 \: \cdot\:) \phi_2 g) ds_1. \\
\end{split}
\]
Applying \eqref{Eqn=gdef} again but now with $\phi_1(s_1 \: \cdot\:) \phi_2$ in place of $\psi$, we get:
\[
\begin{split}
    &T_{\phi_1}(x_1 T_{\phi_2}(x_2 \ldots T_{\phi_n}(x_n)\ldots)) \\
    =& \int_G f_1(s_1) \Delta^{\frac1{p_1}}\l_{s_1}  \int_{G^{\times n-1}} \phi_1(s_1s_2 \ldots s_n) \phi_2(s_2 \ldots s_n) \phi_3(s_3 \ldots s_n) \ldots \phi_n(s_n) \times\\
    & \qquad \qquad \qquad \qquad \qquad \qquad f_2(s_2) \ldots f_n(s_n) \Delta^{\frac1{p_2}} \l_{s_2} \ldots \Delta^{\frac1{p_n}} \l_{s_n} ds_2 \ldots ds_n ds_1 \\
    &= T_{\bar{\phi}}(x_1, \dots, x_n).
\end{split}
\]
\end{proof}

Finally, we calculate a convenient form for the kernel of a corner of the Fourier multiplier for use in Theorem \ref{thm:amenable-intertwining}.

\begin{lemma} \label{Lem=Fourier kernel}
Let $F \subseteq G$ compact and $x_i = \Delta^{a_i} \lambda(f_i) \Delta^{b_i} \in L_{p_i}(\cL G)$ as above for $f_i \in C_c(G) \star C_c(G)$. Then the kernel of $P_F T_\phi(x_1, \dots, x_n) P_F$ is given by
\[
\begin{split}
    (t_0,t_n) \mapsto 1_F(t_0) 1_F(t_n) \int_{G^{\times n-1}}  & \phi(t_0t_1^{-1},\ldots,t_{n-1}t_n^{-1}) f_1(t_0t_1^{-1})\ldots f_n(t_{n-1}t_n^{-1})   \times \\
    & \Delta^{a_1}(t_0) \Delta^{b_1 + a_2}(t_1) \ldots \Delta^{b_n}(t_n) \Delta((t_1 \ldots t_n)^{-1}) dt_1\ldots dt_{n-1}.
\end{split}
\]
\end{lemma}

\begin{proof}
Let $g \in C_c(G) \star C_c(G)$ and $t_0 \in G$. Then the function $P_F T_\phi^{\th_1, \dots, \th_n}(x_1, \dots, x_n) P_F g$ is given by
\[
\begin{split}
    t_0 &\mapsto 1_F(t_0) \int_{G^{\times n}} \phi(t_1,\ldots,t_n) f_1(t_1)\ldots f_n(t_n) \times \\
    & \qquad (\Delta^{a_1} \lambda_{t_1} \Delta^{b_1 + a_2} \lambda_{t_2} \ldots \Delta^{b_{n-1} + a_n} \lambda_{t_n} \Delta^{b_n} P_F g)(t_0) dt_1 \ldots dt_n \\
    &= 1_F(t_0) \int_{G^{\times n}} \phi(t_1,\ldots,t_n) f_1(t_1)\ldots f_n(t_n)  \Delta^{a_1}(t_0) \Delta^{b_1 + a_2}(t_1^{-1}t_0) \times \ldots  \\
    & \qquad \Delta^{b_{n-1} + a_n}(t_{n-1}^{-1} \dots t_1^{-1} t_0) (\Delta^{b_n}1_F g)(t_n^{-1} \dots t_1^{-1} t_0)  dt_1 \ldots dt_n \\
    &= 1_F(t_0) \int_{G^{\times n}} \phi(t_0t_1, t_2, \ldots,t_n) f_1(t_0t_1) f_2(t_2) \ldots f_n(t_n)  \Delta^{a_1}(t_0) \Delta^{b_1 + a_2}(t_1^{-1})  \times  \\
    & \qquad \Delta^{b_2 + a_3}(t_2^{-1} t_1^{-1} ) \ldots \Delta^{b_{n-1} + a_n}(t_{n-1}^{-1} \dots t_1^{-1}) (\Delta^{b_n}1_F g)(t_n^{-1} \dots t_1^{-1})  dt_1 \ldots dt_n \\
    &= 1_F(t_0) \int_{G^{\times n}} \phi(t_0 t_1^{-1}, t_2,\ldots,t_n) f_1(t_0 t_1^{-1}) f_2(t_2)\ldots f_n(t_n)  \Delta^{a_1}(t_0) \Delta^{b_1 + a_2}(t_1) \times  \\
    & \qquad \Delta^{b_2 + a_3}(t_2^{-1} t_1) \ldots \Delta^{b_{n-1} + a_n}(t_{n-1}^{-1} \dots t_2^{-1} t_1) (\Delta^{b_n}1_F g)(t_n^{-1} \dots t_2^{-1} t_1) \Delta(t_1^{-1}) dt_1 \ldots dt_n \\
    &= ... \\
    &= 1_F(t_0) \int_{G^{\times n}} \phi(t_0 t_1^{-1},\ldots,t_{n-1} t_n^{-1}) f_1(t_0 t_1^{-1})\ldots f_n(t_{n-1} t_n^{-1}) \Delta^{a_1}(t_0) \Delta^{b_1 + a_2}(t_1) \times \ldots  \\
    & \qquad \Delta^{b_{n-1} + a_n}(t_{n-1}) (\Delta^{b_n}1_F g)(t_n) \Delta((t_1 \ldots t_n)^{-1}) dt_1 \ldots dt_n. \\
\end{split}
\]
It follows that the kernel has the required form.
\end{proof}

\section{Fourier to Schur transference} \label{Sect=FouriertoSchur}

\noindent Let $G$ be a locally compact first countable group. In this section we prove the transference from Fourier to Schur multipliers for such groups. An important ingredient will be the following `split' coordinate-wise convolution: fix functions $\varphi_k \in C_c(G) \star C_c(G) \subseteq A(G)$ such that $\|\varphi_k\|_1 = 1$ and the supports of $\vphi_k$ form a decreasing neighbourhood basis of $\{e\}$. In other words, $(\vphi_k)$ is an approximate unit for the Banach $*$-algebra $L_1(G)$. Note that we use the first countability of $G$ here. Now, given a bounded function $\phi: G^{\times n} \to \C$ and some fixed $1 \leq i \leq n$, we define 
\[
    \phi_{t_1, \dots, t_n}(s_1, \dots, s_n) = \phi(t_1^{-1} s_1 t_2, \ldots, t_{i-1}^{-1} s_{i-1} t_i, t_i^{-1} s_i, s_{i+1} t_{i+1}^{-1}, t_{i+1} s_{i+2} t_{i+2}^{-1}, \dots, t_{n-1} s_n t_n^{-1})
\]
and
\begin{equation} \label{Eqn=PointwiseConvolution}
\begin{split}
    &\phi_k(s_1, \dots, s_n) := \int_{G^{\times n}} \phi_{t_1, \dots, t_n}(s_1, \dots, s_n) \left( \prod_{j=1}^n \varphi_k(t_j)\right) dt_1 \ldots dt_n \\
    &= \int_{G^{\times n}} \phi_{t_1, \dots, t_n}(e, \dots, e) \left( \prod_{j=1}^i \varphi_k(s_j\ldots s_i t_j) \right) \left( \prod_{j=i+1}^n \varphi_k(t_js_{i+1} \ldots s_j) \Delta(s_{i+1} \ldots s_j) \right) dt_1 \ldots dt_n.
\end{split}
\end{equation}
The last term of \eqref{Eqn=PointwiseConvolution} in combination with Lemma \ref{Lem=Lemma4.4b} will allow us to reduce the problem to the linear case. The necessity of the `split' between indices $i$ and $i+1$ will be explained later. 

\begin{theorem} \label{Thm=FouriertoSchur}
Let $G$ be a locally compact first countable group and let $1 \leq p \leq \infty$, $1<p_1, \ldots, p_{n-1} \leq \infty$ be such that $p^{-1} =  \sum_{i=1}^n p_i^{-1}$.
Let $\phi \in C_b(G^{\times n})$ and define $\widetilde{\phi} \in C_b(G^{\times n + 1})$ by
\begin{equation*} 
    \widetilde{\phi}(s_0, \ldots, s_n) = \phi(s_0 s_1^{-1}, s_1 s_2^{-1}, \ldots, s_{n-1} s_n^{-1}), \qquad s_i \in G.
\end{equation*}
If $\phi$ is defines a $(p_1, \ldots, p_n)$-mb Fourier multiplier $T_\phi$ of $G$, then $\widetilde{\phi}$ defines a $(p_1, \ldots, p_n)$-mb Schur multiplier $M_{\widetilde{\phi}}$ of $G$. Moreover,
\[
\begin{split}
& \Vert M_{\widetilde{\phi}}: S_{p_1}(L_2(G)) \times \ldots \times S_{p_n}(L_2(G)) \rightarrow S_{p}(L_2(G)) \Vert_{(p_1,\ldots,p_n)-mb}\\
& \qquad  \leq
\Vert T_{\phi}: L_{p_1}(\mathcal{L}G  ) \times \ldots \times  L_{p_n}(\mathcal{L}G ) \rightarrow L_{p}(\mathcal{L}G ) \Vert_{(p_1,\ldots,p_n)-mb}.
\end{split}
\]
\end{theorem}

\begin{proof}
Let $F \subseteq G$ finite with $|F| = N$. By Theorem \ref{Thm=Finite truncation}, it suffices to show that the norm of 
\[
M_{\widetilde{\phi}}: S_{p_1}(\ell_2(F)) \times \ldots S_{p_n}(\ell_2(F)) \to S_p(\ell_2(F))
\]
and its matrix amplifications are bounded. \\

For $s \in F$, let $p_s \in B(\ell_2(F))$ be the orthogonal projection onto the span of $\delta_s$. Define the unitary $U = \sum_{s \in F} p_s \otimes \l_s \in B(\ell_2(F)) \otimes \cL G$. In the case $p = \infty$, the Fourier to Schur transference is proven through the transference identity
\[
    T_\phi^{(N)}(U(a_1 \otimes 1)U^*, \ldots, U(a_n \otimes 1)U^*) = U (M_{\widetilde{\phi}}(a_1, \dots, a_n) \otimes 1)U^*, \qquad a_i \in B(\ell_2(F)).
\]
The idea is to do something similar in the case $p < \infty$. However, the unit does not embed in $L_p(\cL G)$, so we need to use some approximation of the unit instead. We construct this as follows: let $\cV = (V_i)_{i \in \N}$ be a decreasing symmetric neighbourhood basis of the identity (this is possible because $G$ is first countable). For $V \in \cV$ we define the operator
\[
    k_V = \|1_V \Delta^{-1/4}\|_2^{-1} \l(1_V \Delta^{-1/4}) \Delta^{1/2} \in L_2(\cL G)
\]
which is proven to be self-adjoint in \cite[Section 8.3]{CPPR}. Let $k_V = u_V h_V$ be its polar decomposition. Then we have $h_V^{2/p} \in L_p(\cL G)$ and, by \eqref{Eqn=PlancherelIdentity}, $\|h_V^{2/p}\|_p = 1$. Now for any $V \in \cV$ we have, by \eqref{Eqn=LpCrossProduct},
\begin{equation} \label{Eqn=TensorWithhV}
    \|M_{\widetilde{\phi}}(a_1, \dots, a_n)\|_{S_p^N} = \|M_{\widetilde{\phi}}(a_1, \dots, a_n) \otimes h_V^{\frac2p}\|_{S_p^N \otimes L_p(\cL G)}, \qquad a_i \in B(\ell_2(F)).
\end{equation}

Next, fix an $i \in \{1, \dots, n\}$ such that $\bar{p}_1 := \left(\sum_{l=1}^i p_l^{-1}\right)^{-1} > 1$ and $\bar{p}_2 := \left(\sum_{l=i+1}^n p_l^{-1}\right)^{-1} > 1$. This is possible by our assumption that $p_1, \dots, p_n > 1$. We now define the functions $\phi_{t_1, \dots, t_n}$ and $\phi_k$ as in \eqref{Eqn=PointwiseConvolution} for the chosen $i$. \\

The condition $\bar{p}_1, \bar{p}_2 > 1$ is necessary for the use of Proposition \ref{Prop=Blackbox} at the end of the proof of Lemma \ref{Lem=Lemma4.6}; this also explains why we need the `split' in the pointwise convolutions. If $p>1$, then one can take $i=n$ in which case the proof of Lemma \ref{Lem=Lemma4.6} simplifies somewhat. Note that in \cite{CJKM} and \cite{CKV}, the convolutions were defined for $i = n-1$. In the latter paper, this creates a problem in case $p_n = \infty, p = 1$; this problem is resolved by splitting instead at some $i$ chosen as above.\\

Let $a_1, \ldots, a_n \in B(\ell_2(F))$. By continuity of $\phi$, we have that $\phi_k \to \phi$ pointwise. Indeed, for $\epsilon > 0$, we can take $K$ such that for $t_1, \dots, t_n \in \supp \vphi_K$, $|\phi_{t_1, \dots, t_n}(s_1, \dots, s_n) - \phi(s_1, \dots, s_n)| < \epsilon$. Then for $k > K$, we get $|\phi_k(s_1, \dots, s_n) - \phi(s_1, \dots, s_n)| < \epsilon$. Since we are working in finite dimensions, this implies
\[
    M_{\widetilde{\phi_k}}(a_1, \dots, a_n) \to M_{\widetilde{\phi}}(a_1, \dots, a_n)
\]
in $S^N_p$. Together with \eqref{Eqn=TensorWithhV} we find
\[
\begin{split}
    \|M_{\widetilde{\phi}}(a_1, \dots, a_n)\|_{S_p^N} &= \lim_k \limsup_{V \in \cV}  \|M_{\widetilde{\phi_k}}(a_1, \dots, a_n) \otimes h_V^{\frac2p}\|_{S_p^N \otimes L_p(\cL G)} \\
    &= \lim_k \limsup_{V \in \cV}  \|U(M_{\widetilde{\phi_k}}(a_1, \dots, a_n) \otimes h_V^{\frac2p})U^*\|_{S_p^N \otimes L_p(\cL G)} \\
    &\leq \limsup_k \limsup_{V \in \cV}  \|T_{\phi_k}^{(N)}(U(a_1 \otimes h_V^{\frac2{p_1}})U^*, \ldots, U(a_n \otimes h_V^{\frac2{p_n}})U^*)\|_{S_p^N \otimes L_p(\cL G)} \\
    &\quad + \limsup_k \limsup_{V \in \cV}  \|T_{\phi_k}^{(N)}(U(a_1 \otimes h_V^{\frac2{p_1}})U^*, \ldots, U(a_n \otimes h_V^{\frac2{p_n}})U^*) \\
    &\qquad \qquad \qquad \qquad \qquad - U(M_{\widetilde{\phi_k}}(a_1, \dots, a_n) \otimes h_V^{\frac2p})U^*\|_{S_p^N \otimes L_p(\cL G)} \\
    &:= A + B.
\end{split}
\]
First, we have
\[
\begin{split}
    A &\leq \limsup_k \limsup_{V \in \cV} \|T_{\phi_k}^{(N)} \| \|a_1 \otimes h_V^{\frac2{p_1}}\|_{S_{p_1}^N \otimes L_{p_1}(\cL G)} \ldots \|a_n \otimes h_V^{\frac2{p_n}}\|_{S_{p_n}^N \otimes L_{p_n}(\cL G)} \\
    &= \limsup_k \|T_{\phi_k}^{(N)}\| \|a_1\|_{S_{p_1}^N} \ldots \|a_n\|_{S_{p_n}^N}.
\end{split}
\]
By repeated use of Lemma \ref{Lem=Lemma 4.3} (in particular we can use Fubini because of the strong continuity property) we find
\[
    \|T_{\phi_k}^{(N)}\| \leq \int_{G^{\times n}} \|T_\phi^{(N)}\| \left( \prod_{i=1}^n |\varphi_k(t_j)| \right) dt_1 \ldots dt_n = \|T_\phi^{(N)}\| \|\varphi_k\|_1^n = \|T_\phi^{(N)}\| \leq \|T_\phi\|_{(p_1, \dots, p_n)-mb}.
\]
and hence
\[
    A \leq \|T_{\phi}\|_{(p_1, \dots, p_n)-mb} \|a_1\|_{S_{p_1}^N} \ldots \|a_n\|_{S_{p_n}^N}.
\]
It remains to show that $B = 0$. By the triangle inequality, it suffices to show this for $a_i = E_{r_{i-1}, r_i}$, $r_0, \dots, r_n \in F$ (for other combinations of matrix units, the term below becomes 0). In that case we get, by applying Lemma \ref{Lem=Lemma 4.3} in the second equality:
\[
\begin{split}
    & T_{\phi_k}^{(N)}(U(E_{r_0, r_1} \otimes h_V^{\frac2{p_1}})U^*, \ldots, U(E_{r_{n-1},r_n} \otimes h_V^{\frac2{p_n}})U^*) - U(M_{\widetilde{\phi_k}}(E_{r_0, r_1}, \dots, E_{r_{n-1},r_n}) \otimes h_V^{\frac2p})U^*\\
    &= E_{r_0, r_n} \otimes \left( T_{\phi_k}(\l_{r_0}h_V^{\frac2{p_1}} \l_{r_1}^*, \ldots, \l_{r_{n-1}}h_V^{\frac2{p_n}} \l_{r_n}^*) - \phi_k(r_0r_1^{-1}, \ldots, r_{n-1}r_n^{-1}) \l_{r_0} h_V^{\frac2p}\l_{r_n}^* \right) \\
    &= E_{r_0, r_n} \otimes \l_{r_0} \left( T_{\phi_k(r_0 \: \cdot\: r_1^{-1}, \ldots, r_{n-1} \: \cdot\: r_n^{-1})}(h_V^{\frac2{p_1}}, \ldots, h_V^{\frac2{p_n}}) - \phi_k(r_0r_1^{-1}, \ldots, r_{n-1}r_n^{-1})  h_V^{\frac2p} \right) \l_{r_n}^*.
\end{split}
\]
Hence,
\begin{equation} \label{Eqn=FinalFormB}
    B = \limsup_k \limsup_{V \in \cV} \left\| T_{\phi_k(r_0\: \cdot\: r_1^{-1}, \ldots, r_{n-1} \: \cdot\: r_n^{-1})}(h_V^{\frac2{p_1}}, \ldots, h_V^{\frac2{p_n}}) - \phi_k(r_0r_1^{-1}, \ldots, r_{n-1}r_n^{-1})  h_V^{\frac2p}  \right\|_{L_p(\cL G)}.
\end{equation}
The limit over $k$ exists and is 0; we postpone the proof to Lemma \ref{Lem=Lemma4.6} below.\\

For the multiplicatively bounded estimate, we prove using similar methods that, for $K \geq 1$ and $a_1, \dots, a_n \in M_K(B(\ell_2(F)))$,
\[
    \|M_{\widetilde{\phi}}^{(K)}(a_1, \dots, a_n)\|_{S_p^{KN}} \leq \|T_\phi^{(KN)}\| \|a_1\|_{S_p^{KN}} \ldots \|a_n\|_{S_p^{KN}}.
\]
Here we use $1_{M_K} \otimes U$ in place of $U$. Moreover, by the triangle inequality it suffices to prove the estimate for $B$ for $a_i = E_{j_{i-1},j_i} \otimes E_{r_{i-1},r_i}$, with $1 \leq j_i \leq K$ and $r_i \in F$; the expression for $B$ then reduces to \eqref{Eqn=FinalFormB} again.
\end{proof}

The following Lemma is similar to \cite[Lemma 4.6]{CJKM}. In our case we have $x_j = 1$ which allows us to avoid the SAIN condition used in that paper; on the other hand, we work with translated functions and our result works for non-unimodular groups. Already in \cite[Theorem 3.1]{CKV} it was explained how to adapt the proof of \cite[Lemma 4.6]{CJKM} for the translated functions. However this paper only considered unimodular groups. Here we spell out the full proof for convenience of the reader.

\begin{lemma} \label{Lem=Lemma4.6}
In the proof of Theorem \ref{Thm=FouriertoSchur}, we have that
\[
\lim_k \limsup_{V \in \cV} \left\| T_{\phi_k(r_0 \: \cdot\: r_1^{-1}, \ldots, r_{n-1} \: \cdot\: r_n^{-1})}(h_V^{\frac2{p_1}}, \ldots, h_V^{\frac2{p_n}}) - \phi_k(r_0r_1^{-1}, \ldots, r_{n-1}r_n^{-1})  h_V^{\frac2p}  \right\|_{L_p(\cL G)} = 0.
\]
\end{lemma}

The main idea is to reduce the problem to the linear case using \eqref{Eqn=PointwiseConvolution} and apply the following result for linear Fourier multipliers:

\begin{proposition} \label{Prop=Blackbox}
    Let $\cV$ be a symmetric neighbourhood basis of the identity of $G$. Let $2 \leq q < p \leq \infty$ or $1 \leq p < q \leq 2$. Assume $\psi \in C_b(G)$ defines a Fourier multiplier on $L_p(\cL G)$. Then we have 
    \[
        \lim_{V \in \cV} \|T_{\psi}(h_V^{2/q}) - \psi(1) h_V^{2/q}\|_{L_q(\cL G)} \to 0.
    \]
\end{proposition}

The proof of Proposition \ref{Prop=Blackbox} is essentially a matter of combining results and remarks from \cite[Proposition 3.9]{CJKM}, \cite[Claim B and Section 8]{CPPR} and applying Haagerup reduction to \cite[Lemma 3.1]{CPR18} to generalise that estimate to general von Neumann algebras. We give more details in Section \ref{Sect=Intertwining}.

\begin{proof}[Proof of Lemma \ref{Lem=Lemma4.6}]
The idea is to use a dominated convergence argument in the last expression of \eqref{Eqn=PointwiseConvolution}. However, the functions $\phi_k$ need not be integrable. We work around this by multiplying with compactly supported functions that are close to 1 around $e$, so that as $V \in \cV$ decreases to $\{e\}$ we are just `multiplying by 1' in the limit. Define a function $\zeta \in C_c(G) \cap A(G)$ with $\zeta(e) = 1$ which is positive definite and (therefore) satisfies $\|T_{\zeta}: L_p(\cL G) \to L_p(\cL G)\| \leq 1$ for all $1 \leq p \leq \infty$. Next let
\[
    \z_j(s) = \z(r_{j-1}^{-1} s r_j), \qquad 1 \leq j \leq n,\ s \in G.
\]
We define a product function as follows:
\[
    (\phi(\z_1, \dots, \z_n))(s_1, \dots, s_n) = \phi(s_1, \dots, s_n) \z_1(s_1) \ldots \z_n(s_n).
\]
Then
\[
\begin{split}
&\| T_{\phi_k(r_0 \: \cdot\: r_1^{-1}, \ldots, r_{n-1} \: \cdot\: r_n^{-1})}(h_V^{\frac2{p_1}}, \ldots, h_V^{\frac2{p_n}}) - \phi_k(r_0r_1^{-1}, \ldots, r_{n-1}r_n^{-1})  h_V^{\frac2p} \|_{L_p(\cL G)} \\
\leq\ & \| (\phi(\z_1, \dots, \z_n))_k (r_0r_1^{-1}, \ldots, r_{n-1}r_n^{-1})  h_V^{\frac2p}- \phi_k(r_0r_1^{-1}, \ldots, r_{n-1}r_n^{-1})  h_V^{\frac2p} \|_{L_p(\cL G)} \\
& + \| T_{(\phi(\z_1, \dots, \z_n))_k(r_0 \: \cdot\: r_1^{-1}, \ldots, r_{n-1} \: \cdot\: r_n^{-1})}(h_V^{\frac2{p_1}}, \ldots, h_V^{\frac2{p_n}}) - (\phi(\z_1, \dots, \z_n))_k(r_0r_1^{-1}, \ldots, r_{n-1}r_n^{-1})  h_V^{\frac2p} \|_{L_p(\cL G)} \\
& + \|T_{\phi_k(r_0 \: \cdot\: r_1^{-1}, \ldots, r_{n-1} \: \cdot\: r_n^{-1})}(h_V^{\frac2{p_1}}, \ldots, h_V^{\frac2{p_n}}) - T_{(\phi(\z_1, \dots, \z_n))_k(r_0 \: \cdot\: r_1^{-1}, \ldots, r_{n-1} \: \cdot\: r_n^{-1})}(h_V^{\frac2{p_1}}, \ldots, h_V^{\frac2{p_n}}) \|_{L_p(\cL G)} \\
=: & A_{k,V} + B_{k,V} + C_{k,V}.
\end{split}
\]
Here $\phi((\z_1, \dots, \z_n))_k$ is defined again by \eqref{Eqn=PointwiseConvolution} for the same $i$. We will estimate these terms separately. We start by showing that $\lim_k \limsup_{V \in \cV} A_{k,V}$ and $\lim_k \limsup_{V \in \cV} C_{k,V}$ are 0, essentially reducing the problem to the integrable functions $\phi(\z_1, \dots, \z_n)$. We then apply the idea mentioned above to show that $\lim_{V \in \cV} B_{k,V} = 0$ for any $k$.\\

Firstly, since $\psi_k \to \psi$ pointwise for any $\psi \in C_b(G)^{\times n}$ we have
\[
\begin{split}
    \limsup_{V \in \cV} A_{k,V} &= |(\phi(\z_1, \dots, \z_n))_k (r_0r_1^{-1}, \ldots, r_{n-1}r_n^{-1}) - \phi_k(r_0r_1^{-1}, \ldots, r_{n-1}r_n^{-1})|\\
    &\to |\phi(r_0r_1^{-1}, \ldots, r_{n-1}r_n^{-1})(1 - \z_1(r_0r_1^{-1}) \ldots \z_n(r_{n-1} r_n^{-1}))| = 0.
\end{split}
\]
Next, we estimate the limit in $k$ of $\limsup_{V \in \cV} C_{k,V}$.
Set $u_j = t_j^{-1} r_{j-1}$, $v_j = t_j r_j$ and 
\[C_V(t_1, \dots, t_n) = \|T_{\eta}(h_V^{\frac2{p_1}}, \dots, h_V^{\frac2{p_n}})\|_{L_p(\cL G)},\]
where
\[
    \eta = (\phi - \phi(\z_1, \dots, \z_n) )(u_1 \cdot u_2^{-1}, \ldots, u_{i-1} \cdot u_i^{-1}, u_i \cdot r_i^{-1}, r_i \cdot v_{i+1}^{-1}, v_{i+1} \cdot v_{i+2}^{-1}, \dots, v_{n-1} \cdot v_n^{-1}).
\]
Thanks to the strong continuity statement of Lemma \ref{Lem=Lemma 4.3}, we can use Fubini to deduce:
\begin{equation} \label{Eqn=CkVEstimate}
C_{k,V} \leq \int_{G^{\times n}} C_V(t_1, \dots, t_n) \left(\prod_{i=1}^n |\varphi_k(t_j)|\right) dt_1 \ldots dt_n.
\end{equation}
Now set
\[
\begin{split}
    &y_{j,V} = \l_{u_j}h_V^{\frac2{p_j}} \l_{u_{j+1}}^*\ \text{for } 1 \leq j \leq i-1, 
    \qquad y_{i,V} = \l_{u_i} h_V^{\frac2{p_i}} \l_{r_i}^*, 
    \qquad \quad y_{i+1,V} = \l_{r_i} h_V^{\frac2{p_{i+1}}} \l_{v_{i+1}}^*, \\
    &y_{j,V} = \l_{v_{j-1}}h_V^{\frac2{p_j}} \l_{v_j}^* \quad \text{for } i+2 \leq j \leq n.
\end{split}
\]
Denote $\iota_q$ for the identity operator on $L_q(\cL G)$. The symbol $1$ is used both for the constant 1-function and the number 1. Then we get the following estimate, where we apply Lemma \ref{Lem=Lemma 4.3} and \eqref{Eqn=Mult-lambda_s} in the first line and Lemma \ref{Lem=Lemma4.4} and the assumption that $T_\zeta$ is a contraction in the third line:
\begin{equation}\label{Eqn=CVEstimate}
\begin{split}
    C_V(t_1, \dots, t_n) &= \|T_{(\phi - \phi(\zeta_1, \dots, \z_n))}(y_{1,V}, \dots, y_{n,V})\|_{L_p(\cL G)}\\
    &\leq \sum_{j=1}^n \|T_{\phi(1, \dots, 1, (\z_j - 1), \zeta_{j+1}, \dots, \z_n)}(y_{1,V}, \dots, y_{n,V}) \|_{L_p(\cL G)} \\
    &\leq \|T_{\phi}: L_{p_1} \times \ldots \times L_{p_n} \to L_p\| \sum_{j=1}^n\left( \|(T_{\z_j} - \iota_{p_j})(y_{j,V})\|_{L_{p_j}(\cL G)} \prod_{i \neq j} \|y_{i,V}\|_{p_i}\right).
\end{split}
\end{equation}
By \eqref{Eqn=Mult-lambda_s}, we have $\|y_{j,V}\|_{p_j} = 1$. Further, by applying again Lemma \ref{Lem=Lemma 4.3} and Proposition \ref{Prop=Blackbox},
\[
\begin{split}
    \|(T_{\zeta_j} - \iota_{p_j})(y_{j,V})\|_{L_{p_j}(\cL G)} &= \|T_{\zeta_j(u_j\: \cdot\: u_{j+1}^{-1}) - 1}(h_V^{\frac2{p_j}})\|_{L_{p_j}(\cL G)} \to |\zeta_j(u_j u_{j+1}^{-1}) - 1|
\end{split}
\]
for $1 \leq j \leq i-1$. Filling in the definition of $\zeta_j$, 
\[
    |\zeta_j(u_j u_{j+1}^{-1}) - 1| = |\zeta(r_{j-1}^{-1} t_j^{-1} r_{j-1} r_j^{-1} t_{j+1} r_j) - 1|
\]
and this equals 0 when evaluated at $t_j, t_{j+1} = e$. Similarly, we find for $i \leq j \leq n$ that $\lim_{V \in \cV} \|(T_{\zeta_j} - \iota)(y_{j,V})\|_{L_{p_j}(\cL G)}$ exists and equals 0 when evaluated at the identity in the corresponding $t_1, \dots, t_n$. Moreover, all these values are bounded by 2. Going back to \eqref{Eqn=CkVEstimate}, let us write $M:=\|T_{\phi}: L_{p_1} \times \ldots \times L_{p_n} \to L_p\|$. We find
\begin{equation} \label{Eqn=CkVEstimate2}
\begin{split}
C_{k,V} \leq& \int_{G^{\times n}} M \left(\sum_{j=1}^n \|(T_{\zeta_j} - \iota)(y_{j,V})\|_{L_{p_j}(\cL G)} \right)  \left(\prod_{j=1}^n |\varphi_k(t_j)| \right) dt_1 \ldots dt_n.
\end{split}
\end{equation}

The integrand of \eqref{Eqn=CkVEstimate2} is bounded by the integrable function $2M\prod_{i=1}^n |\varphi_k(t_j)|$. Hence, by Lebesgue's dominated convergence theorem, the right hand side of \eqref{Eqn=CkVEstimate2} converges in $V$. We find that
\[
\begin{split}
    &\limsup_{V \in \cV} C_{k,V} \leq M \int_{G^{\times n}} \left( \sum_{j=1}^n \lim_{V \in \cV} \|(T_{\zeta_j} - \iota)(y_{j,V})\|_{L_{p_j}(\cL G)} \right)\left(\prod_{i=1}^n |\varphi_k(t_j)|\right)  dt_1 \dots dt_n.
\end{split}
\]
This quantity goes to 0 in $k$. This concludes the proof for $C_{k,V}$. \\

Finally we prove that $\lim_{V \in \cV} B_{k,V} = 0$ for any $k$. We fix a $k$ for the remainder of the proof. Recall that since $\varphi_k \in A(G)$, $T_{\varphi_k}$ is bounded on $L_q(\cL G)$ for any $1 \leq q \leq \infty$. Moreover, since $\vphi_k \in C_c(G) \star C_c(G)$, we also have $\vphi_k \Delta \in C_c(G) \star C_c(G) \subseteq A(G)$ (cf. the calculation before \eqref{Eqn=EqualEmbeddings}), hence $T_{\vphi_k \Delta}$ is also bounded on $L_q(\cL G)$ for any $q$.\\

We may assume, by scaling $\varphi_k$ if necessary, that $T_{\varphi_k}: L_q(\cL G) \to L_q(\cL G)$ and $T_{\varphi_k \Delta}: L_q(\cL G) \to L_q(\cL G)$ are contractions for any H\"older combination $q$ of $p_1, \dots, p_n$. Of course, this means that $\|\varphi_k\|_1$ need no longer be $1$ from now on. Set
\[
\begin{split}
    \psi_k(s_1, \dots, s_n; t_1, \dots, t_n) &:= \left( \prod_{j=1}^i \varphi_k(r_{j-1} s_j\ldots s_{i} r_{i}^{-1} t_j)\right)\\
    & \quad \times \left( \prod_{j=i+1}^n \varphi_k(t_j r_{i} s_{i+1} \ldots s_j r_j^{-1}) \Delta(r_{i} s_{i+1} \ldots s_j r_j^{-1})\right) \\
    &=: \psi_k^1(s_1, \dots, s_i; t_1, \dots, t_i) \psi_k^2(s_{i+1}, \dots, s_n; t_{i+1}, \dots, t_n).
\end{split}
\]
By using the last term of \eqref{Eqn=PointwiseConvolution} and Fubini, we get 
\begin{equation} \label{Eqn=BkVEstimate1}
\begin{split}
    B_{k,V} &\leq \int_{G^{\times n}} |(\phi(\z_1, \dots, \z_n))_{t_1, \dots, t_n}(e, \dots, e)| \\
    & \qquad \times \| T_{\psi_k(\: \cdot \:; t_1, \dots, t_n)}(h_V^{\frac2{p_1}}, \ldots, h_V^{\frac2{p_n}}) - \psi_k(1, \ldots, 1; t_1, \dots t_n) h_V^{\frac2p} \|_{L_p(\cL G)} dt_1 \ldots dt_n.\\
\end{split}
\end{equation}
Note that $|\vphi_k| \leq 1$ by the assumed contractivity of $T_{\varphi_k}$. Indeed, for $s \in G$, apply $T_{\varphi_k}$ to $\l_s$ to deduce that $|\vphi_k(s)| \leq 1$. Hence, $|\psi_k(1, \dots, 1; t_1, \dots, t_n)| \leq \prod_{j=i+1}^n \Delta(r_ir_j^{-1})$.
Moreover, from the expression \eqref{Eqn=TpsiExpression} below we see that $\|T_{\psi_k(\: \cdot \:; t_1, \dots, t_n)}(h_V^{\frac2{p_1}}, \ldots, h_V^{\frac2{p_n}})\|_{L_p(\cL G)} \leq \Delta((t_{i+1}, \dots, t_n)^{-1})$. Since $\phi(\zeta_1, \dots, \zeta_n)$ is compactly supported, the integrand of \eqref{Eqn=BkVEstimate1} is dominated by an integrable function. Hence by the Lebesgue dominated convergence theorem, it suffices to show that the term
\begin{equation} \label{Eqn=PointwiseTerm}
    \| T_{\psi_k(\: \cdot \:; t_1, \dots, t_n)}(h_V^{\frac2{p_1}}, \ldots, h_V^{\frac2{p_n}}) - \psi_k(1, \ldots, 1; t_1, \dots t_n) h_V^{\frac2p} \|_{L_p(\cL G)}
\end{equation}
goes to 0 in $V$ for any choice of $t_1, \dots, t_n \in G$.\\

Fix $t_1, \dots, t_n \in G$. For $1 \leq j \leq i$, set $q_j^{-1} = \sum_{l=j}^{i} p_l^{-1}$ (so $q_1 = \bar{p}_1$) and $T_j = T_{\varphi_k(r_{j-1} \: \cdot \: r_{i}^{-1} t_j)}$. By Lemma \ref{Lem=Lemma 4.3}, $T_j$ is a contraction on $L_{q_j}(\cL G)$. For $i+1 \leq j \leq n$, set $q_j^{-1} = \sum_{l=j}^n  p_l^{-1}$ (so $q_{i+1} = \bar{p}_2$) and $T_j = T_{\varphi_k(t_j r_i\: \cdot \: r_j^{-1}) \Delta(r_i \: \cdot \: r_j^{-1})}$. We can estimate the norm of $T_j: L_{q_j}(\cL G) \to L_{q_j}(\cL G)$ by using again Lemma \ref{Lem=Lemma 4.3}:
\[
    \|T_j\| = \Delta(t_j^{-1}) \|T_{\vphi_k(t_j r_i\: \cdot\: r_j^{-1}) \Delta(t_j r_i\: \cdot\: r_j^{-1})}\| = \Delta(t_j^{-1}) \|T_{\vphi_k \Delta}\| \leq \Delta(t_j^{-1}).
\]
Now, by Lemma \ref{Lem=Lemma4.4b}, we have
\[
\begin{split}
    T_{\psi_k^1(\cdot; t_1, \dots, t_i)}(h_V^{\frac2{p_1}}, \dots, h_V^{\frac2{p_i}}) &= T_1(h_V^{\frac2{p_1}} T_2(h_V^{\frac2{p_2}} \dots T_i(h_V^{\frac2{p_i}}) \dots )), \\
    T_{\psi_k^2(\cdot; t_{i+1}, \dots, t_n)}(h_V^{\frac2{p_{i+1}}}, \dots, h_V^{\frac2{p_n}}) &= T_{i+1}(h_V^{\frac2{p_{i+1}}} T_{i+2}(h_V^{\frac2{p_{i+2}}} \dots T_n(h_V^{\frac2{p_n}}) \dots )).
\end{split}
\]
Clearly, $T_{\psi_k^1(\cdot; t_1, \dots, t_i)}$ is contractive as a map on $L_{p_1}(\cL G) \times \ldots \times L_{p_i}(\cL G)$. Let $x_j \in L_{p_j}(\cL G)$ with $\|x_j\|_{L_{p_j}(\cL G)} \leq 1$; then, from \eqref{Eqn=FourierProducts},
\begin{equation} \label{Eqn=TpsiExpression}
\begin{split}
    \|T_{\psi_k(\: \cdot\: ; t_1, \dots, t_n)}(x_1, \dots, x_n)\|_{L_p(\cL G)} &\leq \|T_{\psi^2_k(\: \cdot\: ; t_1, \dots, t_n)}(x_{i+1}, \dots, x_n)\|_{L_{\bar{p}_2}(\cL G)} \\
    &\leq \Delta((t_{i+1} \ldots t_n)^{-1}).
\end{split}
\end{equation}
This validates the use of the dominated convergence theorem above. Now we go back to estimating \eqref{Eqn=PointwiseTerm}. Using subsequently the triangle inequality and H\"older's inequality (with again \eqref{Eqn=FourierProducts}), we find
\[
\begin{split}
    &\| T_{\psi_k(\: \cdot \:; t_1, \dots, t_n)}(h_V^{\frac2{p_1}}, \ldots, h_V^{\frac2{p_n}}) - \psi_k(1, \ldots, 1; t_1, \dots t_n) h_V^{\frac2p} \|_{L_p(\cL G)} \\
    &\leq \|T_{\psi^1_k(\: \cdot \:; t_1, \dots, t_i)}(h_V^{\frac2{p_1}}, \ldots, h_V^{\frac2{p_i}}) \cdot \psi_k^2(1, \dots, 1; t_{i+1}, \dots, t_n) h_V^{\frac2{\bar{p}_2}} - \psi_k(1, \ldots, 1; t_1, \dots t_n) h_V^{\frac2p} \|_{L_p(\cL G)} \\
    &\quad + \| T_{\psi_k(\: \cdot \:; t_1, \dots, t_n)}(h_V^{\frac2{p_1}}, \ldots, h_V^{\frac2{p_n}}) - T_{\psi^1_k(\: \cdot \:; t_1, \dots, t_i)}(h_V^{\frac2{p_1}}, \ldots, h_V^{\frac2{p_i}}) \cdot \psi_k^2(1, \dots, 1; t_{i+1}, \dots, t_n) h_V^{\frac2{\bar{p}_2}}\|_{L_p(\cL G)}\\
    &\leq \left(\prod_{j=i+1}^n \Delta(r_ir_j^{-1})\right) \|T_{\psi^1_k(\: \cdot \:; t_1, \dots, t_i)}(h_V^{\frac2{p_1}}, \ldots, h_V^{\frac2{p_i}}) - \psi_k^1(1, \dots, 1; t_1, \dots, t_i) h_V^{\frac2{\bar{p}_1}}\|_{L_{\bar{p}_1}(\cL G)} \\
    &\quad + \|T_{\psi_k^2(\cdot; t_{i+1}, \dots, t_n)}(h_V^{\frac2{p_{i+1}}}, \dots, h_V^{\frac2{p_n}}) - \psi_k^2(1, \dots, 1; t_{i+1}, \dots, t_n) h_V^{\frac2{\bar{p}_2}} \|_{L_{\bar{p}_2}(\cL G)} \\
    & =: B_{k,V}^1 + B_{k,V}^2.
\end{split}
\]
We show only that $\lim_{V \in \cV} B_{k,V}^2 = 0$; the equality $\lim_{V \in \cV} B_{k,V}^1 = 0$ follows similarly and is in fact slightly easier since the $T_j$ are contractions for $j \leq i$. Now set, for $i \leq j \leq n$,
\[
    R_{j,V} := \left(\prod_{l=j+1}^{n} \varphi_k(t_l r_i r_l^{-1}) \right) T_{i+1}(h_V^{\frac2{p_{i+1}}} \ldots T_j( h_V^{\frac2{q_j}}) \ldots ).
\]
Here $R_{i,V} = \prod_{l=i+1}^{n-1} \varphi_k(t_l r_i r_l^{-1}) h_V^{\frac2{q_1}}$. Then 
\[
    B_{k,V}^2 \leq \sum_{j=i+1}^n \|R_{j,V} - R_{j-1, V}\|_{L_{\bar{p}_2}(\cL G)}.
\]
Recall that $|\vphi_k| \leq 1$. Hence
\[
\begin{split}
    &\|R_{j,V} - R_{j-1,V}\|_{L_{\bar{p}_2}(\cL G)} \\
    &= \left(\prod_{l=j+1}^n |\varphi_k(t_l r_i r_l^{-1})|\right) \|T_{i+1}(h_V^{\frac2{p_{i+1}}} \ldots T_{j-1}(h_V^{\frac2{p_{j-1}}}(T_j(h_V^{\frac2{q_j}}) - \varphi_k(t_j r_i r_j^{-1}) h_V^{\frac2{q_j}})) \ldots )\|_{L_{\bar{p}_2}(\cL G)} \\
    &\leq \Delta((t_{i+1} \ldots t_n)^{-1}) \|T_j(h_V^{\frac2{q_j}}) - \vphi_k(t_j r_i r_j^{-1}) h_V^{\frac2{q_j}}\|_{L_{q_j}(\cL G)}.
\end{split}
\]
We know that $q_j > \bar{p}_2 > 1$ for any $i+1 \leq j \leq n$. Additionally, $T_j$ is bounded on $\cL G$ and $L_1(\cL G)$. By Proposition \ref{Prop=Blackbox}, the above terms converge to 0 in $V$. Hence, $\lim_{V \in \cV} B_{k,V}^2 = 0$. This finishes the proof. 
\end{proof}

\begin{remark}
As in the unimodular case (see \cite[Remark 3.3]{CKV}) we do not know if Theorem \ref{Thm=FouriertoSchur} holds if $p = p_i = 1$ for some $1 \leq i \leq n$ (and $p_j = \infty$ for all $j \neq i$). The proof above fails in that case because we cannot apply Proposition \ref{Prop=Blackbox}.
\end{remark}

\section{Schur to Fourier transference for amenable groups} \label{Sect=SchurToFourier}

In this section we extend \cite[Proof of Theorem 4.1]{CKV}, i.e. the transference from Schur multipliers to Fourier multipliers for amenable groups, to the non-unimodular setting. The proof is essentially the same, but with extra technicalities due to the modular function. We also fixed a small mistake in the proof, as will be mentioned at the relevant spot.\\

Recall \cite[Theorem 4.10]{Paterson} that $G$ is amenable iff it satisfies the following F\o lner condition: for any $\epsilon>0$ and any compact set $K\subseteq G$, there exists a compact set $F \subseteq G$ with non-zero measure such that $\frac{\mu((sF \setminus F) \cup (F \setminus sF))}{\mu(F)} < \epsilon$ for all $s\in K$. This allows us to construct a net $F_{(\epsilon,K)}$ of such F\o lner sets using the ordering $( \epsilon_1, K_1) \leq (\epsilon_2,K_2)$ if $\epsilon_1 \geq \epsilon_2, K_1 \subseteq K_2$.

\begin{theorem}\label{thm:amenable-intertwining}
Let $G$ be a locally compact, amenable group and let $1\leq p,p',p_1,\ldots,p_n \leq \infty$ be such that $\frac1p=\sum_{i=1}^n \frac1{p_i} = 1 - \frac1{p'}$. Let $\phi \in L_\infty(G^{\times n})$ and define $\widetilde{\phi} \in  L_\infty(G^{\times n + 1})$ by
\begin{equation*}
    \widetilde{\phi}(s_0, \ldots, s_n) = \phi(s_0 s_1^{-1}, s_1 s_2^{-1}, \ldots, s_{n-1} s_n^{-1}), \qquad s_i \in G.
\end{equation*}
Assume that $\widetilde{\phi}$ defines a $(p_1,\ldots,p_n)$-Schur multiplier of $G$. Then there is a net $I$ and there are complete contractions $i_{q,\alpha}: L_q(\cL G) \to  S_q(L_2(G))$, $\alpha \in I$, such that for all $f_i,f \in C_c(G)\star C_c(G)$,
\begin{equation}\label{Eqn=AmenableIntertwining}
\left|\langle i_{p,\alpha}(T_\phi(x_1,\ldots,x_n)),i_{p',\alpha}(y) \rangle -  \langle M_{\widetilde{\phi}}( i_{p_1,\alpha}(x_1),\ldots,i_{p_n,\alpha}(x_n) ) , i_{p',\alpha}(y) \rangle \right| \stackrel{\alpha}\to 0,
\end{equation}
where $x_i = \Delta^{a_i}\lambda(f_i)\Delta^{b_i} \in L^{p_i}(\cL G),\ y = \Delta^a \lambda(f) \Delta^b \in L^{p'}(\cL G)$ (i.e. $a_i + b_i = \frac1{p_i}$). In a similar way, the matrix amplifications of $i_{q,\alpha}$ approximately intertwine the multiplicative amplifications of the Fourier and Schur multipliers.
\end{theorem}

\begin{proof}
Let $F_\alpha, \alpha \in I$ be a F\o lner net for $G$, where $I$ is the index set consisting of pairs $(\epsilon, K)$ for $\epsilon > 0$, $K \subseteq G$ compact and the ordering as described above. \\

Let $P_\alpha=P_{F_\alpha}$ be the projection of $L_2(G)$ onto $L_2(F_\alpha)$. Consider the maps $i_{p,\alpha}: L_p(\cL G)\to S_p(L_2(G))$ defined by $i_{p,\alpha}(x) = \mu(F_\alpha)^{-1/p} P_{\alpha}x P_{\alpha}$. They are contractions by \cite[Theorem 5.1]{CaspersDeLaSalle}. By replacing $G$ by $G \times SU(2)$, one proves that they are in fact complete contractions (see also the last paragraph of \cite[Proof of Theorem 5.2]{CaspersDeLaSalle}).\\

Now fix $\alpha$. From \eqref{Eqn=Kernel of single operator}, we deduce
\begin{align*}
M_{\widetilde{\phi}}  (i_{p_1, \alpha}(x_1),\ldots,i_{p_n, \alpha}(x_n)) (t_0,t_{n})  \\ =
\frac{1}{\mu(F_\alpha)^{1/p}} 1_{F_\alpha}(t_0)1_{F_\alpha}(t_n) \int_{F_\alpha^{\times n-1}} &\phi(t_0t_1^{-1},\ldots, t_{n-1}t_n^{-1}) f_1(t_0t_1^{-1})\ldots f_n(t_{n-1}t_n^{-1}) \times \\
& \Delta^{a_1}(t_0) \Delta^{b_1 + a_2}(t_1) \ldots \Delta^{b_n}(t_n) \Delta((t_1 \ldots t_n)^{-1}) dt_1 \ldots dt_{n-1}.
\end{align*}
From Lemma \ref{Lem=Fourier kernel} we have a similar expression for the kernel of $i_{p,\alpha}(T_\phi(x_1, \dots, x_n))$:
\[
\begin{split}
    (t_0,t_n) \mapsto \frac1{\mu(F_{\alpha})^{1/p}} 1_{F_{\alpha}}(t_0) 1_{F_{\alpha}}(t_n) & \int_{G^{\times n-1}} \phi(t_0t_1^{-1},\ldots,t_{n-1}t_n^{-1}) f_1(t_0t_1^{-1})\ldots f_n(t_{n-1}t_n^{-1})   \times \\
    & \Delta^{a_1}(t_0) \Delta^{b_1 + a_2}(t_1) \ldots \Delta^{b_n}(t_n) \Delta((t_1 \ldots t_n)^{-1}) dt_1\ldots dt_{n-1}.
\end{split}
\]\

Now we need to take the pairing of these kernels with $i_{p',\alpha}(y)$ and calculate their difference. To that end, we define the following function $\Phi$:
\[
\begin{split}
 \Phi(t_0,\ldots,t_n) = &\phi(t_0t_1^{-1},\ldots,t_{n-1}t_n^{-1}) f_1(t_0t_1^{-1})\ldots f_n(t_{n-1}t_n^{-1}) f(t_n t_0^{-1}) \times \\
 &  \Delta^{a_1+b}(t_0) \Delta^{b_1 + a_2}(t_1) \ldots \Delta^{b_n+a}(t_n) \Delta((t_0t_1 \ldots t_n)^{-1}),
\end{split}
\]
and the function $\Psi_{\alpha}$:
\[
\Psi_\alpha(t_0,\ldots, t_n) = 1_{F_\alpha}(t_0) 1_{F_\alpha} (t_n) - 1_{F_\alpha^{\times n+1}}(t_0,\ldots, t_n) = 1_{F_\alpha \times (F_{\alpha}^{\times n-1})^c \times F_{\alpha}}(t_0, \dots, t_n).
\]
Note that in \cite{CKV}, the indicator function was mistakenly taken over $F_\alpha \times (F_\alpha^c)^{\times n-1} \times F_{\alpha}$ instead. This correction leads to an extra term $n$ in the choice of the lower bound of $\alpha$ at the end. Also note that a priori, it is not clear that $T_\phi(x_1, \dots, x_n)$ lies in $L_p(\cL G)$, and hence it is not clear that $i_{p,\alpha}(T_\phi(x_1, \dots, x_n))$ lies in $S_p(L_2(G))$. However, both $i_{p,\alpha}(T_\phi(x_1, \dots, x_n))$ and $i_{p',\alpha}(y)$ are given by integration against a kernel in $L_2(G \times G)$, so the pairing \eqref{Eqn=KernelsDualPairing} is still well-defined as a pairing in $S_2(L_2(G))$ instead. Now we have:
\begin{equation} \label{Eqn=IntertwiningStep1}
\begin{split}
&\vert \langle i_{p,\alpha}(T_\phi(x_1,\ldots,x_n)),i_{p',\alpha}(y) \rangle - \langle M_{\widetilde{\phi}}( i_{p_1,\alpha}(x_1),\ldots,i_{p_n,\alpha}(x_n) )  , i_{p',\alpha}(y) \rangle \vert \\
&= \left\vert \frac{1}{\mu(F_\alpha)} \int_{G^{\times n+1}} \Phi(t_0,\ldots,t_n) \Psi_\alpha(t_0,\ldots,t_n) dt_0 \ldots dt_n \right\vert \\
\end{split}
\end{equation}

Let $K \subseteq G$ be some compact set such that $\supp(f_j),\supp(f)\subseteq K$ and $e \in K$.  Let $t_0, \dots, t_n$ be such that both $\Phi(t_0, \dots, t_n)$ and $\Psi_\alpha(t_0, \dots, t_n)$ are nonzero. Since $\Psi_\alpha(t_0, \dots, t_n)$ is nonzero, we must have $t_0, t_n \in F_\alpha$ and $t_i \notin F_\alpha$ for some $i \in \{1, \dots, n-1\}$. Since $\Phi(t_0, \dots, t_n)$ is nonzero, there are $k_1,\ldots,k_n \in K$ such that $t_{n-1} = k_n t_n,\ t_{n-2} = k_{n-1}k_n t_n,\ \ldots  ,\ t_0 = k_1 \ldots k_n t_n$. Hence we find 
\begin{equation} \label{Eqn=SetManipulations}
\begin{split}
t_n &\in F_\alpha \cap (k_1\ldots k_n)^{-1}F_\alpha \setminus \left( (k_2\ldots k_n)^{-1}F_\alpha \cap \ldots \cap k_n^{-1}F_\alpha \right) \\
&\subseteq F_\alpha \setminus \left( (k_2\ldots k_n)^{-1}F_\alpha \cap \ldots \cap k_n^{-1}F_\alpha \right) \\
&= (F_\alpha \setminus (k_2\ldots k_n)^{-1}F_\alpha) \cup \ldots \cup (F_\alpha \setminus k_n^{-1} F_\alpha)
\end{split}
\end{equation}
We want to apply change of variables in \eqref{Eqn=IntertwiningStep1}. Let us first look at a simple case: assume $g \in L_1(G \times G)$ is such that $g(s,t) \neq 0$ only when $st^{-1} \in K$. Then
\[
\begin{split}
    \int_{G^{\times 2}} g(s,t) ds dt &= \int_{G^{\times 2}} 1_K(s t^{-1}) g(s, t) ds dt 
    = \int_{G^{\times 2}} 1_K(s) g(s t, t) \Delta(t) ds dt \\
    &= \int_G \int_K g(k_1 t, t) \Delta(t) dk_1 dt \\
\end{split}
\]
where we renamed the variable $s$ in the last line. Applying the above formula twice for a function $g \in L_1(G^{\times 3})$ such that $g(r,s,t) \neq 0$ only when $rs^{-1} \in K$, $st^{-1} \in K$, we get
\[
\begin{split}
    \int_{G^{\times 3}} g(r,s,t) dr ds dt &= \int_{G^{\times 2}} \int_K g(k_1 s, s, t) \Delta(s) dk_1 ds dt \\
    &= \int_G \int_{K^{\times 2}} g(k_1k_2 t, k_2t, t) \Delta(k_2t) \Delta(t) dk_1 dk_2 dt.
\end{split}
\]

Carrying on like this, we obtain
\begin{equation} \label{Eqn=IntertwiningStep2}
\begin{split}
    &\left\vert \frac{1}{\mu(F_\alpha)} \int_{G^{\times n+1}} \Phi(t_0,\ldots,t_n) \Psi_\alpha(t_0,\ldots,t_n) dt_0 \ldots dt_n \right\vert \\
    &= \bigg\vert \frac{1}{\mu(F_\alpha)} \int_{K^{\times n}} \int_{G} \Phi(k_1\ldots k_n t_n, \ldots , k_n t_n, t_n)  \Psi_\alpha(k_1\ldots k_nt_n,\ldots ,k_nt_n, t_n) \times \\
    & \qquad \Delta(k_2 \dots k_n t_n) \dots \Delta(k_nt_n) \Delta(t_n) dt_n dk_1 \ldots dk_n  \bigg\vert. \\
\end{split}
\end{equation}
Note that $a + b + \sum_{i=1}^n a_i + b_i = 1$, hence
\begin{equation} \label{Eqn=IntertwiningStep3}
\begin{split}
    &|\Phi(k_1\ldots k_n t_n, \ldots , k_n t_n, t_n)|\Delta(k_2 \dots k_n t_n) \dots \Delta(k_nt_n) \Delta(t_n)  \\
    &\leq \|\phi f_1 \dots f_n f\|_{\infty} \Delta^{a_1+b}(k_1\ldots k_n t_n) \Delta^{b_1+a_2+1}(k_2 \ldots k_n t_n) \ldots \Delta^{b_{n-1} + a_n+1}(k_nt_n) \times \\
    & \qquad \Delta^{b_n + a+1}(t_n)\Delta(k_1^{-1} k_2^{-2} \dots k_n^{-n} t_n^{-n-1}) \\
    &= \|\phi f_1 \dots f_n f\|_{\infty} \Delta^{a_1+b-1}(k_1) \Delta^{a_1+a_2+b_1+b-1}(k_2) \ldots \Delta^{1-b_n-a-1}(k_n)  \\
    &\leq \|\phi f_1 \dots f_n f\|_{\infty} C_{K,n} =: M.
\end{split}
\end{equation}
Here the constant $C_{K,n}$ can be chosen to be dependent only on $K$ and $n$ (and $G$).\\

Applying \eqref{Eqn=IntertwiningStep1}, \eqref{Eqn=IntertwiningStep2}, \eqref{Eqn=IntertwiningStep3} and \eqref{Eqn=SetManipulations} consecutively we get
\begin{equation} \label{Eqn=UltraIntertwining}
\begin{split}
&\vert \langle i_{p,\alpha}(T_\phi(x_1,\ldots,x_n)),i_{p',\alpha}(y) \rangle_{p,p'} - \langle M_{\widetilde{\phi}}( i_{p_1,\alpha}(x_1),\ldots,i_{p_n,\alpha}(x_n) )  , i_{p',\alpha}(y) \rangle \vert \\
&= \bigg\vert \frac{1}{\mu(F_\alpha)} \int_{K^{\times n}} \int_{G} \Phi(k_1\ldots k_n t_n, \ldots , k_n t_n, t_n)  \Psi_\alpha(k_1\ldots k_nt_n,\ldots ,k_nt_n, t_n) \times \\
    & \qquad \Delta(k_2 \dots k_n t_n) \dots \Delta(k_nt_n) \Delta(t_n) dt_n dk_1 \ldots dk_n  \bigg\vert \\
&\leq \frac{M}{\mu(F_\alpha)} \int_{K^{\times n}} \int_G \Psi_\alpha(k_1\ldots k_nt_n,\ldots ,k_nt_n, t_n) dt_n dk_1 \ldots dk_n \\
&= \frac{M}{\mu(F_\alpha)} \int_{K^{\times n}}  \mu\left(F_\alpha \cap (k_1\ldots k_n)^{-1}F_\alpha \setminus \left( (k_2\ldots k_n)^{-1}F_\alpha \cap \ldots \cap k_n^{-1}F_\alpha \right)\right) dk_1 \ldots dk_n\\
&\leq \frac{M}{\mu(F_\alpha)} \int_{K^{\times n}} \sum_{i=2}^n \mu(F_\alpha \setminus (k_i \ldots k_n)^{-1}F_\alpha)  dk_1 \ldots dk_n\\
& \leq M (n-1) \mu(K)^n \sup_{k \in K^{1-n}} \frac{\mu(F_{\alpha} \setminus kF_{\alpha})} {\mu(F_\alpha)}.
\end{split}
\end{equation}

Using the ordering described earlier, if the index $\alpha \geq (\epsilon \times \left(Mn \mu_G(K)^n\right)^{-1},  K^{1-n})$, then the F\o lner condition implies that \eqref{Eqn=UltraIntertwining} is less than $\epsilon$, and hence the limit \eqref{Eqn=AmenableIntertwining} holds.\\

From \eqref{Eqn=AmenableIntertwining}, it follows from writing out the definitions that the matrix amplifications of $i_{p,\alpha}$ also approximately intertwine the multiplicative amplifications of the Fourier and Schur multipliers. i.e. for $\beta_i \in S_{p_i}^N, \beta \in S_{p'}^{N}$, we have 

\begin{equation}
\begin{split}
 \Big|\langle \id \otimes & i_{p,\alpha}(T^{(N)}_\phi(\beta_1\otimes x_1,\ldots,\beta_n \otimes x_n)),  \id\otimes i_{p',\alpha}(\beta \otimes y) \rangle_{p,p'} - \\
   &\langle M^{(N)}_{\widetilde{\phi}}(\id\otimes i_{p_1,\alpha}(\beta_1 \otimes x_1),\ldots,\id\otimes i_{p_n,\alpha}(\beta_n \otimes x_n) ) , \id\otimes i_{p',\alpha}(\beta\otimes y) \rangle \Big| \to 0
  \end{split}
  \end{equation}
\end{proof}

\begin{corollary} \label{Cor=Schur to Fourier}
Let $G$ be an amenable locally compact group and $1 \leq p, p_1, \dots, p_n \leq \infty$ be such that $\frac1p=\sum_{i=1}^n \frac1{p_i}$. Let $\phi \in  L_\infty(G^{\times n})$. If $\widetilde{\phi}$ is the symbol of a $(p_1, \dots, p_n)$-bounded (resp. multiplicatively bounded) Schur multiplier then $\phi$ is the symbol of a $(p_1, \dots, p_n)$-bounded (resp. multiplicatively bounded) Fourier multiplier. Moreover, 
\[
    \|T_\phi\|_{(p_1, \dots, p_n)} \leq \|M_{\widetilde{\phi}}\|_{(p_1, \dots, p_n)}, \qquad \|T_\phi\|_{(p_1, \dots, p_n)-mb} \leq \|M_{\widetilde{\phi}}\|_{(p_1, \dots, p_n)-mb}.
\]
\end{corollary}

\begin{proof}
Let $x_i$ be as in the hypotheses of Theorem \ref{thm:amenable-intertwining} and let $i_{p,\alpha}$ be as in the proof of Theorem \ref{thm:amenable-intertwining}. Let $p'$ be the Holder conjugate of $p$. In \cite[Theorem 5.2]{CaspersDeLaSalle} it is proven that 
\begin{equation} \label{Eqn=ApproxIsometry}
    \la i_{p,\alpha}(x), i_{p', \alpha}(y) \ra_{p,p'}\to \la x, y \ra_{p,p'}, \qquad x \in L_p(\cL G),\ y \in L_{p'}(\cL G).
\end{equation}
Note that this inequality also holds, and in fact is explicitly proven for $p = \infty$; by symmetry it also holds for $p=1$. We remark that this result also uses the F\o lner condition.\\

Let $\epsilon > 0$. Then we can find $y = \Delta^{a} \lambda(f) \Delta^b$, $f \in C_c(G) \star C_c(G)$ such that $\|y\|_{L_{p'}(\cL G)} \leq 1$ and 
\[
    \|T_{\phi}(x_1, \dots, x_n)\|_{L_p(\cL G)} \leq |\la T_{\phi}(x_1, \dots, x_n), y \ra| + \epsilon.
\]
Next, by \eqref{Eqn=ApproxIsometry} and Theorem \ref{thm:amenable-intertwining} we can find $\alpha \in I$ such that the following two inequalities hold:
\[
    |\la T_{\phi}(x_1, \dots, x_n), y \ra - \la i_{p,\alpha}(T_{\phi}(x_1, \dots, x_n)), i_{p', \alpha}(y) \ra| < \epsilon
\]
and
\[
    \vert \langle i_{p,\alpha}(T_\phi(x_1,\ldots,x_n)),i_{p',\alpha}(y) \rangle_{p,p'} - \langle M_{\widetilde{\phi}}( i_{p_1,\alpha}(x_1),\ldots,i_{p_n,\alpha}(x_n) )  , i_{p',\alpha}(y) \rangle_{p,p'} \vert < \epsilon.
\]
By combining these inequalities we find
\[
\begin{split}
    \|T_{\phi}(x_1, \dots, x_n)\|_{L_p(\cL G)} &\leq |\langle M_{\widetilde{\phi}}( i_{p_1,\alpha}(x_1),\ldots,i_{p_n,\alpha}(x_n) )  , i_{p',\alpha}(y) \rangle_{p,p'}| + 3\epsilon \\
    &\leq \|M_{\widetilde{\phi}}\|_{(p_1, \dots, p_n)} \prod_{i=1}^n \|x_i\|_{L_{p_i}(\cL G)} + 3\epsilon
\end{split}
\]
The elements $x_i$ as chosen above are norm dense in $L_{p_i}(\cL G)$ (resp. $C^*_\lambda(G)$ when $p_i = \infty$), hence we get the required bound. The multiplicative bound follows similarly.
\end{proof}

\begin{remark}
In \cite{CaspersDeLaSalle}, \cite{CKV}, the proof runs via an ultraproduct construction. The ultraproduct is not actually necessary as demonstrated above, as all limits are usual limits and not ultralimits.
\end{remark}

We can now extend the result of \cite[Corollary 4.3]{CKV} to non-unimodular groups. It is also a multiplicatively bounded, non-unimodular version of \cite[Theorem 4.5]{CJKM}. Moreover, we no longer need the SAIN condition and the subgroup need no longer be discrete.

\begin{corollary}
Let $G$ be a locally compact, first countable group and let $1 \leq p \leq \infty$ and $1 < p_1, \dots, p_n \leq \infty$ with $p^{-1} = \sum_{i=1}^n p_i^{-1}$. Let $\phi \in C_b(G^{\times n})$ which defines a $(p_1, \dots, p_n)$-mb Fourier multiplier and let $H \leq G$ be an amenable subgroup. Then
\[
    \|T_{\phi|_{H^{\times n}}}\|_{(p_1, \dots, p_n)-mb} \leq \|T_{\phi}\|_{(p_1, \dots, p_n)-mb}
\]
\end{corollary}

\begin{proof}
    The associated inequality for Schur multipliers follows from Theorem \ref{Thm=Finite truncation}. Now Corollary \ref{Cor=Schur to Fourier} (using amenability of $H$) and Theorem \ref{Thm=FouriertoSchur} yield the result.
\end{proof}

In the next corollary we prove a necessary condition for a `Fourier multiplier' to satisfy \eqref{Eqn=AmenableIntertwining} for the embeddings $i_{p,\alpha}$ defined above. This was used in the discussion in Section \ref{Sect=FourierDef}.

\begin{corollary} \label{Cor=NecessaryCondition}
Fix $n > 1$, $1 \leq p_1, \dots, p_n, p, p' \leq \infty$ such that $\frac1p = \sum_{i=1}^n \frac1{p_i} = 1 - \frac1{p'}$ and let $\th_1, \dots, \th_n, \th, \th' \in [0,1]$. Let $i_{p,\alpha}$ be as in the proof of Theorem \ref{thm:amenable-intertwining}. Assume that for each $\phi \in L_\infty(G^{\times n})$, we have a map $S_\phi: \k^{\th_1}_{p_i}(L) \times \ldots \times \k^{\th_n}_{p_n}(L) \to \k^{\th}_p(L)$ satisfying 
\begin{equation}\label{Eqn=CorollaryAmenableIntertwining}
\left|\langle i_{p,\alpha}(S_\phi(x_1,\ldots,x_n)),i_{p',\alpha}(y) \rangle_{p,p'} -  \langle M_{\widetilde{\phi}}( i_{p_1,\alpha}(x_1),\ldots,i_{p_n,\alpha}(x_n) ) , i_{p',\alpha}(y) \rangle_{p,p'} \right| \stackrel{\alpha}\to 0.
\end{equation}
for $x_i \in \k_{p_i}^{\th_i}(L)$, $y \in \k_{p'}^{\th'}(L)$.
Now let $\phi(s_1, \dots, s_n) = \phi_1(s_1)\ldots \phi_n(s_n)$ for some functions $\phi_1, \dots, \phi_n \in L^\infty(G)$. Then $S_\phi$ must satisfy
\[
    S_{\phi}(x_1, \dots, x_n) = T_{\phi_1}(x_1) \ldots T_{\phi_n}(x_n)
\]
for $x_i \in \k^{\th_i}_{p_i}(L)$, $i = 1, \dots, n$.
\end{corollary}

\begin{proof}
Fix some $y \in \k^{\th'}_{p'}(L)$. By density, it suffices to show that
\[
    \la S_\phi(x_1, \dots, x_n), y\ra =  \la T_{\phi_1}(x_1) \ldots T_{\phi_n}(x_n), y\ra.
\]
By \eqref{Eqn=ApproxIsometry}, it suffices to show
\[
    \lim_{\alpha \in I} |\la i_{p,\alpha}(S_\phi(x_1, \dots, x_n) - T_{\phi_1}(x_1) \ldots T_{\phi_n}(x_n)), i_{p',\alpha}(y)\ra| = 0.
\]
By running the proof of Theorem \ref{thm:amenable-intertwining} with the constant 1 function in place of $\phi$ and $T_{\phi_i}(x_i)$ in place of $x_i$, we find that
\[
    \lim_{\alpha \in I} |\la i_{p,\alpha}(T_{\phi_1}(x_1) \ldots T_{\phi_n}(x_n)) - i_{p_1, \alpha}(T_{\phi_1}(x_1)) \ldots i_{p_n, \alpha}(T_{\phi_n}(x_n)), i_{p',\alpha}(y) \ra| = 0.
\]
Since multiplication with $\phi_i$ only maps $C_c(G) \star C_c(G)$ to $C_c(G)$, we no longer need to have that $T_{\phi_i}(x_i) \in \k_{p_i}^{\th_i}(L)$, so we cannot apply Theorem \ref{thm:amenable-intertwining} directly. But since we have $\phi = 1$, this does not give any technical complications in the proof. \\

\noindent Using the kernel representations, it is straightforward to show that
\[
\begin{split}
    i_{p_1, \alpha}(T_{\phi_1}(x_1)) \ldots i_{p_n, \alpha}(T_{\phi_n}(x_n)) &= M_{\widetilde{\phi_1}}(i_{p_1, \alpha}(x_1)) \ldots M_{\widetilde{\phi_n}}(i_{p_n,\alpha}(x_n)) \\
    &= M_{\widetilde{\phi}}(i_{p_1, \alpha}(x_1), \dots, i_{p_n,\alpha}(x_n)).
\end{split}
\]
By combining the above observations with \eqref{Eqn=CorollaryAmenableIntertwining}, we get the required result.
\end{proof}

\section{Linear intertwining result} \label{Sect=Intertwining}

In this section, we sketch the proof of Proposition \ref{Prop=Blackbox}. The main ingredient to be added to already existing results is the extension of \cite[Lemma 3.1]{CPR18} to general von Neumann algebras via Haagerup reduction. The Haagerup reduction method is described by Theorem \ref{Thm=HaagerupReduction}, proved for $\s$-finite von Neumann algebras in \cite{HaagerupReduction} and extended to the weight case in \cite[Section 8]{CPPR}. We will assume that the reader is familiar with Tomita-Takesaki theory, conditional expectations and such. We refer to \cite{Takesaki2} for the background.\\

Denote by $\s^\vphi$ the modular automorphism group of a normal faithful semifinite (nfs) weight $\vphi$. Recall that the centraliser $\cN_\vphi$ of a nfs weight $\vphi$ on a von Neumann algebra $\cM$ is given by 
\[ \cN_\vphi = \{x \in \cM: \s_t^\vphi(x) = x\ \forall t \in \R\}. \]

\begin{theorem} \label{Thm=HaagerupReduction}
Let $(\cM, \vphi)$ be any von Neumann algebra equipped with a nfs weight. There is another von Neumann algebra $(\cR, \widehat{\vphi})$ containing $\cM$ and with nfs weight $\widehat{\vphi}$ extending $\vphi$, and elements $a_n$ in the center of the centralizer of $\widehat{\vphi}$ such that the following properties hold:
\begin{enumerate}
    \item There is a conditional expectation $\cE: \cR \to \cM$ satisfying
    \[ \vphi \circ \cE = \widehat{\vphi}, \qquad \s_s^{\vphi} \circ \cE = \cE \circ \s_s^{\widehat{\vphi}}, \quad s \in \R.\]
    \item The centralisers $\cR_n$ of the weights $\vphi_n := \vphi(e^{-a_n}\: \cdot\: )$ are semifinite for $n \geq 1$.
    \item There are conditional expectations $\cE_n: \cR \to \cR_n$ satisfying
    \[ \widehat{\vphi} \circ \cE_n = \widehat{\vphi}, \qquad \s_s^{\widehat{\vphi}} \circ \cE_n = \cE_n \circ \s_s^{\widehat{\vphi}}, \quad s \in \R\]
    \item $\cE_n(x) \to x$ $\s$-strongly for $x \in \fn_{\widehat\vphi}$, and $\bigcup_{n \geq 1} \cR_n$ is $\s$-strongly dense in $\cR$. 
\end{enumerate}
\end{theorem}

We denote by $D_\vphi$ the spatial derivative with respect to $\vphi$ (and some weight on the commutant, whose choice is unimportant). Assume that $T: \cM \to \cM$ is unital completely positive (ucp) and satisfies $\vphi \circ T \leq \vphi$. Then by \cite[Section 5]{HaagerupReduction}, $T$ `extends' to a map $T^{(p)}$ on $L_p(\cM)$, in the sense that $T^{(p)}(D_\vphi^{1/2p} x D_\vphi^{1/2p}) = D_\vphi^{1/2p} T(x) D_\vphi^{1/2p}$ for $x \in \fm_{\widehat\vphi}$. If $T$ satisfies $\s_s^\vphi \circ T = T \circ \s_s^\vphi$, then we moreover have $T^{(p)}(D_\vphi^{\theta/p} x D_\vphi^{(1-\theta)/p}) = D_\vphi^{\theta/p} T(x) D_\vphi^{(1-\theta)/p}$ for any $0 \leq \theta \leq 1$ and $x \in \fm_{\widehat\vphi}$. \\

\noindent In particular, the conditional expectations $\cE, \cE_n$ `extend' to maps $\cE^{(p)}, \cE_n^{(p)}$ from $L_p(\cR, \widehat{\vphi})$ to $L_p(\cM, \vphi)$ resp. $L_p(\cR_n, \widehat{\vphi})$. The following statement is \cite[Lemma 8.3]{CPPR}:
\begin{equation}\label{Eqn=HaagerupReduction}
    \lim_{n \to \infty}\|\cE_n^{(p)}(x) - x\|_p = 0, \qquad 1 \leq p < \infty,\ x \in L_p(\cR, \hat{\vphi}).
\end{equation}
\noindent We need a few more facts; we refer to \cite[Section 8.2]{CPPR} for the details. First, there is an isometric isomorphism $\k_p: L_p(\cR_n, \widehat{\vphi}) \to L_p(\cR_n, \vphi_n)$ given by $\k_p(D_{\widehat{\vphi}}^{1/2p}xD_{\widehat\vphi}^{1/2p}) = e^{a_n/2p} x e^{a_n/2p}$ for $x \in \fm_{\widehat\vphi}$. Next, assume that $T: \cM \to \cM$ is ucp and preserves $\vphi$ and $\s_s^\vphi$. Then by \cite[Section 4]{HaagerupReduction} there exists an extension $\widehat{T}: \cR \to \cR$ which is also ucp and preserves $\widehat{\vphi}$ and $\s_s^{\widehat\vphi}$. Hence $\widehat{T}$ itself also `extends' to the various noncommutative $L_p$-spaces. Moreover, the following diagram commutes:
\[
\begin{tikzcd}
L_p(\cR, \widehat\vphi) \arrow{r}{\cE_n^{(p)}} & L_p(\cR_n, \widehat\vphi) \arrow{r}{\k_p} & L_p(\cR_n, \vphi_n)\\
L_p(\cR, \widehat\vphi) \arrow{u}{\widehat{T}^{(p)}} \arrow{r}{\cE_n^{(p)}} & L_p(\cR_n, \widehat\vphi) \arrow{u}{\widehat{T}^{(p)}} \arrow{r}{\k_p} & L_p(\cR_n, \vphi_n) \arrow{u}{\widehat{T}}.
\end{tikzcd}
\]
Note that since $\vphi_n$ is tracial on $\cR_n$, the $\widehat{T}$ in the rightmost upwards arrow is actually an extension of the operator $\widehat{T}$ on $\cR_n$ so we do not need to use the notation $\widehat{T}^{(p)}$ here.\\

Finally, for $1 \leq p,q < \infty$ we define the Mazur maps $M_{p,q}: L_p(\cM) \to L_q(\cM)$ by $x \mapsto u |x|^{p/q}$ where $x = u|x|$ is the polar decomposition of $x$. The Mazur maps satisfy $\k_q \circ M_{p,q} = M_{p,q} \circ \k_p$, see for instance \cite[end of Section 3]{RicardMazur}. We are now ready to state and prove the generalisation of \cite[Lemma 3.1]{CPR18} for general von Neumann algebras. This result was already shown for $2 < p < \infty$ in \cite[Section 8]{CPPR}, but we will prove the result for all $1 < p < \infty$ at once since this does not take any extra effort.

\begin{lemma} \label{Lem=AlmostMultMaps}
Let $(\cM, \vphi)$ be a von Neumann algebra equipped with nfs weight. Let $T: \cM \to \cM$ be a unital completely positive map satisfying $\vphi \circ T = \vphi$ and $T \circ \s_s^\vphi = \s_s^\vphi \circ T$ for all $s \in \R$. Then there exists a universal constant $C>0$ such that for any $x \in L_2(\cM)$ and $1 < p < \infty$:
\[
    \|T^{(p)}(M_{2,p}(x)) - M_{2,p}(x)\|_p \leq C \|T^{(2)}(x) - x\|_2^\theta \|x\|_2^{1-\theta},
\]
where $\theta = \frac14 \min\{\frac p2, \frac2p\}$.
\end{lemma}

\begin{proof}
The proof runs via Haagerup reduction, using the estimates for the semifinite case from \cite[Claim B]{CPPR} for $p > 2$ and \cite[Lemma 3.1]{CPR18} for $p < 2$. Note that the latter was stated only for finite von Neumann algebras, but the same proof works for the semifinite case as well. \\

Set $y = M_{2,p}(x)$. Since $T = \widehat{T}$ on $\cM$ and $L_p(\cM, \vphi) \hookrightarrow L_p(\cR, \widehat\vphi)$ canonically and isometrically, we have $T^{(p)}(y) = \widehat{T}^{(p)}(y)$ and 
\[ \|T^{(p)}(y) - y\|_{L_p(\cM, \vphi)} = \|\widehat{T}^{(p)}(y) - y\|_{L_p(\cR, \widehat\vphi)}.\]
Now fix $n \geq 1$. Then
\[
\begin{split}
    \|\cE_n^{(p)}(\widehat{T}^{(p)}(y)) - \cE_n^{(p)}(y)\|_{L_p(\cR_n, \widehat\vphi)} &= \|\k_p\left(\cE_n^{(p)}(\widehat{T}^{(p)}(y)) - \cE_n^{(p)}(y)\right)\|_{L_p(\cR_n, \vphi_n)} \\
    &= \|\widehat{T}(\k_p(\cE_n^{(p)}(y))) - \k_p(\cE_n^{(p)}(y))\|_{L_p(\cR_n, \vphi_n)}.
\end{split}
\]
Now we can apply the result for the semifinite case on $\k_p(\cE_n^{(p)}(y))$ to obtain
\[
\begin{split}
    &\|\cE_n^{(p)}(\widehat{T}^{(p)}(y)) - \cE_n^{(p)}(y)\|_{L_p(\cR_n, \widehat\vphi)} \\
    \leq\ & C \|\hat{T}(M_{p,2}(\k_p(\cE_n^{(p)}(y)))) - M_{p,2}(\k_p(\cE_n^{(p)}(y)))\|_{L_2(\cR_n, \vphi_n)}^\theta \cdot \|M_{p,2}(\k_p(\cE_n^{(p)}(y)))\|_{L_2(\cR_n, \vphi_n)}^{1-\theta} \\
    =\ & C \|\k_2(\hat{T}^{(2)}(M_{p,2}(\cE_n^{(p)}(y)))) - \k_2(M_{p,2}(\cE_n^{(p)}(y)))\|_{L_2(\cR_n, \vphi_n)}^\theta \cdot \|\k_2(M_{p,2}(\cE_n^{(p)}(y)))\|_{L_2(\cR_n, \vphi_n)}^{1-\theta} \\
    =\ & C \|\hat{T}^{(2)}(M_{p,2}(\cE_n^{(p)}(y))) - M_{p,2}(\cE_n^{(p)}(y))\|_{L_2(\cR_n, \widehat\vphi)}^\theta \cdot \|M_{p,2}(\cE_n^{(p)}(y))\|_{L_2(\cR_n, \widehat\vphi)}^{1-\theta} \\
    =:\ & C A_n^\theta B_n^{1-\theta}.\\
\end{split}
\]
By the triangle inequality, the main result from \cite{RicardMazur} and \eqref{Eqn=HaagerupReduction}, we find
\[
\begin{split}
    B_n &\leq \|M_{p,2}(\cE_n^{(p)}(y)) - M_{p,2}(y)\|_{L_2(\cR_n, \widehat\vphi)} + \|M_{p,2}(y)\|_{L_2(\cR_n, \widehat\vphi)} \\
    &\leq C_{x,p} \|\cE_n^{(p)}(y) - y\|_{L_p(\cR_n, \widehat\vphi)}^{\min\{\frac p2, 1\}} + \|x\|_{L_2(\cR_n, \widehat\vphi)} \to \|x\|_{L_2(\cM, \vphi)}
\end{split}
\]
for some constant $C_{x,p}$ independent of $n$. Similarly, we find
\[
A_n \leq C_{x,p}\|\widehat{T}^{(2)} - 1_{\cR}\| \|\cE_n^{(p)}(y) - y\|_{L_p(\cR_n, \widehat\vphi)}^{\min\{\frac p2, 1\}} + \|\widehat{T}^{(2)}(x) - x\|_{L_2(\cR_n, \widehat\vphi)} \to \|T^{(2)}(x) - x\|_{L_2(\cM, \vphi)}.
\]
Hence, taking limits and applying again \eqref{Eqn=HaagerupReduction}, we conclude
\[
    \|T^{(p)}(y) - y\|_p = \lim_{n \to \infty} \|\cE_n^{(p)}(\widehat{T}^{(p)}(y)) - \cE_n^{(p)}(y)\|_{L_p(\cR_n, \widehat\vphi)} \leq C \|T^{(2)}(x) - x\|_2^\theta \|x\|_2^{1-\theta}.
\]
\end{proof}

\begin{proof}[Proof of Proposition \ref{Prop=Blackbox}]
We indicate only the changes to \cite[Proof of Claim B]{CPPR}. The statement we have to prove is precisely \cite[Equation (9)]{CPPR}, but without the $u_j$ (this is just a different choice based on convenience).
The $T_\zeta$ constructed in \cite[Proof of Claim B]{CPPR} is a $\vphi$-preserving ucp map that commutes with the modular automorphism group; one can see this from \eqref{Eqn=CommutationFormula2}. Hence, we can apply Lemma \ref{Lem=AlmostMultMaps} on $T_\zeta$ and $h_V$ to show \cite[Equation (10)]{CPPR} (but without the $u_j$). Then, setting $z_j = h_V^{2/q}$, the rest of the proof is the same. 
\end{proof}

\textbf{Acknowledgement.} The author thanks Martijn Caspers for useful discussions and a thorough proofreading of the manuscript.

\end{document}